\documentclass[12pt,reqno,a4paper]{amsart}

\usepackage{amsmath,amsfonts,amssymb,bm}

\usepackage{cancel}
\usepackage{tikz-cd}
\usepackage{ulem}
\def\lcf{\lbrack\! \lbrack}
\def\rcf{\rbrack\! \rbrack}
\usepackage[latin1]{inputenc}
\usepackage{float} 
\usepackage{subfig}
\usepackage{graphicx} 
\usepackage[all]{xy}
\usepackage{graphicx}

\usepackage{hyperref}
\hypersetup{
	colorlinks   =  true,
	linkcolor    = cyan,
	citecolor    = red,
	urlcolor	=magenta,     
}
\usepackage{amsmath}
\usepackage{amssymb} 
\usepackage{tikz}

\makeatletter
\newcommand{\xleftrightarrow}[2][]{\ext@arrow 3359\leftrightarrowfill@{#1}{#2}}
\makeatother

\newcommand{\xdasharrow}[2][->]{
	\tikz[baseline=-\the\dimexpr\fontdimen22\textfont2\relax]{
		\node[anchor=south,font=\scriptsize, inner ysep=1.5pt,outer xsep=2.2pt](x){#2};
		\draw[shorten <=3.4pt,shorten >=3.4pt,dashed,#1](x.south west)--(x.south east);
	}
}
\usepackage{amsthm}

\newtheorem{definition}{Definition}[section]
\newtheorem{lemma}[definition]{Lemma}
\newtheorem{theorem}[definition]{Theorem}
\newtheorem{proposition}[definition]{Proposition}
\newtheorem{corollary}[definition]{Corollary}
\newtheorem{remark}[definition]{Remark}
\newtheorem{example}[definition]{Example}

\newtheorem{comment}[definition]{Comment}

\newcommand{\longsquiggly}{\xymatrix{{}\ar@{~>}[r]&{}}}
\newcommand{\sT}{\mathsf{T}}
\newcommand{\rs}{\mathtt r}

\newcommand{\dsn}{{{\delta}}^{n}}
\newcommand{\ads}{\underline{\ad}}
\newcommand{\dn}{{{\delta}}_{^tn}}
\newcommand{\adn}{{\ad}^{n}}
\newcommand{\adt}{{\ad}^{(2)}}
	{\nolinebreak \hfill $\Box$ \end{trivlist}}

\newcommand\gl{\mathfrak {gl}}
\renewcommand\sl{\mathfrak {sl}}
\newcommand\so{\mathfrak {so}}
\newcommand\ad{\mathrm{ad}}

\def \g{\mathfrak{g}}
\def \R{\mathbb{R}}

\renewcommand\d{\mathrm{d}}

\thanks{AMS Mathematics Subject Classification (2020): 53D17, 17B62, 37K10.}
\thanks{Keywords: Poisson-Nijenhuis structures, Lie bialgebras and coboundary Lie bialgebras,  Poisson-Lie groups, completely integrable systems.  }

\begin{document}
	
	\title[NL bialgebras]{NL bialgebras}
	
	\author[Z. Ravanpak]{Zohreh Ravanpak}
	\address{Z.\ Ravanpak: 
		Departamentul de Matematic\u a, Universitatea de Vest din Timi\c soara \\} 
	\email{zohreh.ravanpak@e-uvt.ro}

	\begin{abstract} In this paper, we introduce the concept of (weak) NL bialgebras. These structures consist of a Lie bialgebra  $(\g,[\cdot,\cdot],\delta)$ equipped with a Nijenhuis structure on the Lie algebra $(\g,[\cdot,\cdot])$, satisfying specific compatibility conditions.  This construction is analogous to Poisson-Nijenhuis structures studied in the context of integrable systems. We further investigate NL bialgebras that generate a compatible hierarchy of bialgebras, both on the original Lie algebra and its deformed versions, through the Nijenhuis structure of any order. Additionally, we demonstrate that the underlying algebraic structure of a particular case of the Euler-top system is a weak NL bialgebra.
	\end{abstract}
		\maketitle
	\tableofcontents
	\section{Introduction }
	 Poisson-Nijenhuis (PN) structures on manifolds were introduced by Magri and Morosi \cite{MaMo} (see also \cite{KoMa}).  These structures generate a compatible hierarchy of Poisson structures through a recursion operator and are significant in integrable systems theory.  More precisely, they generalize the concept of bi-Hamiltonian systems, providing a rich framework for the exploration of integrability of Hamiltonian dynamics.
	 
	 	 In this paper, we study this approach within the context of Lie bialgebras. Specifically, given a Lie bialgebra $(\g,[\cdot,\cdot],\delta)$ with a Nijenhuis operator on the Lie algebra $(\g,[\cdot,\cdot])$, we analyze the compatibility conditions required to establish a compatible hierarchy of Lie bialgebras, both on the original Lie algebra and on iteratively deformed Lie brackets, utilizing the Nijenhuis structure of every order. We introduce the term (weak) NL bialgebra, to describe this structure.
	 	 
These structures contribute to understanding the interplay between Nijenhuis operators and Lie bialgebras. They reveal how compatibility conditions facilitate the construction of a series of compatible Lie bialgebras. This exploration enhances the understanding of the algebraic framework underlying integrable systems. It contributes to the broader field of mathematical physics by establishing connections between these algebraic structures and their applications in various dynamical systems.
	
Some Lie bialgebras, referred to as coboundary Lie bialgebras, are derived from solutions of the classical Yang-Baxter equation on a Lie algebra. These solutions are significant in the theory of integrable systems, establishing a compelling link between integrable systems and Poisson-Lie groups  \cite{Ti1, Ti2}. Moreover, we explore the interplay between coboundary Lie bialgebras and Nijenhuis structures, revealing their significance as examples of NL bialgebras.
	
There are numerous studies in the literature that examine the interplay between Nijenhuis structures and other mathematical frameworks, such as: \cite{Grab,HaZo, Ko3, RaReHa, Sti} .

The infinitesimal counterpart of PN groupoids are introduced in the work by A.~Das \cite{Das}. They are characterized by Lie bialgebroids with a linear endomorphism, together with the linear Poisson structure on the dual Lie algebroid, which form again a PN structure. However, our approach is not primarily geometric; rather, we aim to explore the hierarchy of deformations from a fundamentally algebraic perspective, starting with the cohomology of Lie algebras. 

Our perspective on NL bialgebras also differs from the concept of Lie-Nijenhuis bialgebroids introduced in \cite{Dru}, where the latter describes Lie bialgebroids with a generalized derivation of degree one providing an infinitesimal explanation of PN groupoids. NL bialgebras, on the other hand, are characterized within a distinct algebraic framework, focusing on $1$-cocycles in the cohomology of Lie algebras defined by deformed Lie brackets. 
	
It is well-known that for a PN structure on a manifold $M$, the tangent bundle $\sT M$ equipped with the Lie algebroid bracket $[\cdot,\cdot]_N$ and the cotangent bundle $\sT^* M$ equipped with the Lie algebroid bracket $[\cdot,\cdot]_\Pi$ form a Lie bialgebroid \cite{Ko4, Mack}. In the theory of PN structures, the Nijenhuis structure is considered on the tangent bundle of the ambient space where the PN structure resides. In contrast, for NL bialgebras, the Nijenhuis structure is defined directly on the Lie algebra itself. In our case, given an (almost) NL bialgebra $(\g,[\cdot,\cdot],\delta,n)$, one can apply the Kirillov-Kostant-Souriau (KKS) Poisson bracket on $\g^*$. Consequently, the structure $((\g,[\cdot,\cdot]_n), (\g^*,[\cdot,\cdot]_{\g^*}))$ forms a Lie algebroid over a point.

 We remark that the theory of deformations of Lie algebras, see  (\cite{Ger,Ni1}), and of Nijenhuis deformations (N-deformations) of Lie algebras in our paper are related, but distinct concepts. In fact, the brackets $[\cdot,\cdot]_n$, which are deformed by a linear operator $n$ on $\mathfrak{g}$ in our work, represent trivial infinitesimal deformations in the context of the general theory of deformations of Lie algebras. Specifically, we have
\[
[\xi_1,\xi_2]_t=[\xi_1,\xi_2]+t[\xi_1,\xi_2]_n+(\mbox{higher order terms}), \,  \quad t\in \R
\,,\]
where $[\xi_1,\xi_2]_t$ defines a Lie algebra structure on $\g$ isomorphic to that defined by $[\xi_1,\xi_2]$. Indeed, $[\cdot,\cdot]_n$ is a coboundary, $[\cdot,\cdot]_n=\partial n$, in the Chevalley-Eilenberg cochain complex with values in $\g$, and thus it corresponds to a zero cohomology class in $ H^2(\mathfrak{g}, \mathfrak{g})$. 
 In Subsection \ref{sec6},  we demonstrate that the semi-simple real Lie algebra $\so(3)$ can be non-trivially $N$-deformed---not in the general sense of deformation of the Lie algebras--- into the Lie algebra $\sl(2,\R)$ using an almost Nijenhuis structure. We will maintain the convention established by Kosmann-Schwarzbach and Magri, as noted in their work \cite{KoMa}, to refer to $[\cdot,\cdot]_n$---with a slight abuse of terminology---as a deformed Lie bracket when $n$ is (almost) Nijenhuis.

		Our motivation arises from the importance of bi-Hamiltonian systems within the context of Poisson-Lie groups. The infinitesimal counterpart of a Poisson-Lie group is a Lie bialgebra. Conversely, the adjoint one-cocycle associated with a Lie bialgebra induces a unique Poisson structure that is compatible with the group multiplication on a connected, simply connected Lie group that integrates the Lie algebra. Consequently, a compatible hierarchy of Lie bialgebras on a Lie algebra leads to a family of compatible multiplicative Poisson structures on the corresponding Lie group. These compatible Poisson structures are essential for constructing bi-Hamiltonian systems, which, under certain conditions, are completely integrable \cite{KoMa}. 
		
	Consequently, (weak) NL bialgebras play a significant role in the theory of integrable systems. In particular, the structures arising from (weak) NL bialgebras can be used to construct integrable models in classical mechanics. The added flexibility provided by the Nijenhuis operator facilitates the development of new integrable systems and solutions. Furthermore, they generalize the concept of bi-Hamiltonian systems to a class where the algebraic structures are (weak) NL bialgebras.

		The structure of this paper is as follows: Section \ref{sec2} reviews definitions and fundamental results on Poisson-Nijenhuis structures, Lie bialgebras, and coboundary Lie bialgebras derived from solutions to the classical Yang-Baxter equation. In Section \ref{sec3}, given a Lie bialgebra $(\g,[\cdot,\cdot],\delta)$ and a Nijenhuis structure on the Lie algebra  $(\g,[\cdot,\cdot])$, we introduce three types of deformation for the corresponding 1-cocycle. Subsequently, we define almost NL bialgebras and (weak) NL bialgebras $(\g,[\cdot,\cdot],\delta,n)$. Section \ref{sec4} delves into coboundary Lie bialgebras obtained from solutions of the classical Yang-Baxter equation, with a focus on the role of the Nijenhuis operator. In Section \ref{sec5}, we discuss the hierarchy of structures and present additional results on NL bialgebras, which lead to a compatible hierarchy of Lie bialgebras. We apply our findings to a well-known dynamical system, a particular case of the Euler-top,  demonstrating that the underlying algebraic structure of $\so(3)$ Euler-top dynamics is a weak NL bialgebra. Finally, a concluding section closes the paper.

	\section{ Poisson-Nijenhuis structures and (coboundary) Lie bialgebras}\label{sec2}
	In this section, we will review fundamental definitions and results about Poisson-Nijenhuis structures on manifolds, Lie bialgebras and solution of the classical Yang-Baxter equation (for more details, see \cite{Ko, KoMa}).
	\subsection{Poisson-Nijenhuis structures on Manifolds}
	A Poisson-Nijenhuis structure $(\Pi,\mathrm{N})$ on a manifold $M$ consists of
	\begin{itemize}
		\item	A Poisson structure on $M$, i.e. a bivector field $\Pi$ on $M$ for which
		$$	[\Pi, \Pi] = 0,$$
		where $[\cdot, \cdot ]$ is the Schouten-Nijenhuis bracket. The Poisson structure $\Pi$ induces a vector bundle map from the cotangent bundle $\mathsf T^*M$ of $M$ on the tangent bundle $\mathsf T M$ of $M$, which we will denote by $\Pi^{\sharp}: \mathsf T^*M \to \mathsf T M$, given by
		\[\langle\alpha, \Pi^{\sharp}(\beta) \rangle = \Pi(\alpha, \beta), \; \; \mbox{ for } \alpha, \beta \in \mathsf T_x^*M  \mbox{ and } x\in M.\] 
		
		\item A Nijenhuis structure on $M$, i.e. a $(1,1)$-tensor field $\mathrm{N}: \mathsf T M \to \mathsf T M $ whose Nijenhuis torsion $\lcf \mathrm{N}, \mathrm{N}\rcf $ is zero. $\lcf \mathrm{N}, \mathrm{N}\rcf $ is a $(1,2)$-tensor field on $M$ which is defined as 
		\[
		\lcf \mathrm{N}, \mathrm{N}\rcf (X, Y) = [\mathrm{N}X, \mathrm{N}Y] - \mathrm{N}[\mathrm{N}X, Y] - \mathrm{N}[X, \mathrm{N}Y] + \mathrm{N}^2[X, Y],\quad \forall X, Y \in {\mathfrak X}(M)\,. \]
	\end{itemize}
	
	\noindent  satisfying two compatibility conditions 
	\[
\mathrm{N} \circ \Pi^{\sharp} = \Pi^{\sharp} \circ{ }^t\mathrm{N}\,,
	\]
	\begin{equation}\label{Def-Concomi}
		\begin{array}{rcl}
			C(\Pi,\mathrm{N})(\alpha,\beta)&=& \mathcal{L}_{\Pi^{\sharp}\alpha}({ }^t\mathrm{N}\beta)-\mathcal{L}_{\Pi^{\sharp}\beta}({ }^t\mathrm{N}\alpha)+{ }^t\mathrm{N}\mathcal{L}_{\Pi^{\sharp}\beta}\alpha-{ }^t\mathrm{N}\mathcal{L}_{\Pi^{\sharp}\alpha}\beta\\[4pt]
			&&+~ \d \left\langle \alpha,\mathrm{N}\Pi^{\sharp}\beta\right\rangle
			+\mathrm{N}^{t}\d\left\langle \alpha,\Pi^{\sharp}\beta\right\rangle,\quad \alpha, \beta \in \Omega^1(M)\,.
		\end{array}
	\end{equation}
	
	\noindent Here, ${ }^t \mathrm{N}$ is the dual $(1,1)$-tensor field to $\mathrm{N}$ and $C(\Pi, \mathrm{N})$ is the so-called the concomitant of $\Pi$ and $\mathrm{N}$. The concomitant $C(\Pi, \mathrm{N})$ is a $(2,1)$-tensor field on $M$.
	
	An important fact is that with a Poisson-Nijenhuis   structure $(\Pi, \mathrm{N})$ on $M$ one may produce a hierarchy of Poisson structures $\Pi_k$, $k \in \mathbb{N} \cup \{0\}$ (with $\Pi_0 = \Pi$) which are compatible, that is,
	\[
	[\Pi_k, \Pi_l] = 0, \; \; \mbox{ for } k, l \in \mathbb{N}.
	\]
	The Poisson structure $\Pi_k$ in this hierarchy is characterized by the condition
	\[
	\Pi_k^{\sharp} = \mathrm{N}^k \circ \Pi^{\sharp}, \; \; \mbox{ for } k \in \mathbb{N}, \; \; k \geq 1.
	\]
	Indeed, the pair $(\Pi, \mathrm{N}^k)$ is a Poisson-Nijenhuis structure on $M$.
	
	\subsection{Lie bialgebras} Before defining a Lie bialgebra, it is essential to review some fundamental concepts in the cohomology of Lie algebras.
    
	Let $(\mathfrak g,[\cdot, \cdot])$ be a finite dimensional Lie algebra over the field of real number and consider a representation  of $\mathfrak g$  on a vector space $V,$ i.e. an $\R$-bilinear map 	$\cdot: \mathfrak g\times V\to V$  such that
	$$ [\xi_1,\xi_2]\cdot v=(\xi_1\cdot(\xi_2\cdot v))-(\xi_2\cdot(\xi_1\cdot v)),\mbox{ for all }\xi_i\in {\mathfrak g} \mbox{ and } v\in V.$$
	
	\noindent	The cohomology associated with this representation on the Lie algebra $\mathfrak g$  is defined as follows:
	
	If $k$ is an arbitrary nonnegative integer, the space $C^k(\mathfrak g,V)$ of {\it $k$-cochains} of $\mathfrak g$ with values on $V$ is the set of  skew-symmetric $k$-linear maps $c_k: \mathfrak g\times \dots  \times \mathfrak g\to V.$ The {\it $0$-cochains} are just the elements of $V$. The coboundary operator $\partial:C^k(\mathfrak g,V) \to C^{k+1}(\mathfrak g,V)$ of this cohomology  is the linear map given by 
	\begin{equation}\label{partial}
		\begin{array}{rcl}
			{(\partial c_k)}(\xi_0,...\xi_k)&=&\displaystyle\sum _{i=0}^{k} (-1)^{i}\xi_i \cdot c_k (\xi_0,...,\hat{\xi_i}
			,...,\xi_k)\\[8pt] &&+\displaystyle\sum_{\substack{ i,j=0 \\ i<j}}^k(-1)^{i+j}c_k([\xi_i,\xi_j],\xi_0,...,\hat {\xi_i},...,\hat{\xi_j},...,\xi_k),
		\end{array}
	\end{equation}
	where $c_k$ is a $k$-cochain  and $\xi_0,...,\xi_k \in \mathfrak g$. It is straightforward to prove that  $\partial^2=0$ and therefore $(C^\bullet(\mathfrak g,V), \partial)$ forms a cohomology complex.
	
	\noindent Two specific cases of this construction that we address in this paper are:
	\begin{itemize}
		\item  The algebraic cohomology of a Lie algebra $(\mathfrak g,[\cdot, \cdot])$, which is defined as the cohomology associated with the trivial representation of $\mathfrak g$ on $\R$. In this case the $1$-cochains correspond to the elements of $\mathfrak g^*$. If $\d$ is the differential for this cohomology, then,  for all $\alpha\in \mathfrak g^*$ and $\xi_i\in \mathfrak g,$  $\d\alpha(\xi_1,\xi_2)=-\alpha([\xi_1,\xi_2]).$ 
		
		\item Let  $V:=\wedge^p\mathfrak g$ be the vector space of the skew-symmetric $n$-linear maps  on $\mathfrak g^*$. We consider  the extension $\ad^{(p)}:\g \times \wedge^p \g \to \wedge^p \mathfrak g$  of the adjoint representation  $\ad$ of  $\mathfrak g$ to the space $\wedge^p{\mathfrak g}$.  For all $\xi\in \mathfrak g$ and $\eta_i\in \mathfrak g^*$, it is defined by  
		\begin{equation}\label{ad2}
			{\ad^{(p)}_\xi(P)(\eta_1,\dots, \eta_k)=\lcf \xi, P\rcf(\eta_1,\dots ,\eta_p)}= \sum_{i=1}^pP(\eta_1,\dots, \ad^*_{\xi}\eta_i,\dots \eta_p)\,, \quad P\in \wedge^p \mathfrak g\,.
		\end{equation}
 Here $\lcf \cdot,\cdot\rcf$ is the algebraic Schouten-Nijenhuis bracket for the Lie algebra $(\mathfrak g,[\cdot,\cdot]).$ 		
	\end{itemize}				
\begin{definition}	
A $k$-cochain $\delta$ is called a $k$-cocycle if $\partial \delta=0$. 
\end{definition}

	\noindent For instance, a $1$-cochain $\delta:\mathfrak g\to \wedge^p\mathfrak g$ in the cohomology induced by the representation $\ad^{(p)}$ is a $1$-cocycle if
	\begin{equation}\label{cocycle}
		\partial \delta(\xi_1,\xi_2)=	\ad_{\xi_1}^{(p)}(\delta \xi_2)-\ad_{\xi_2}^{(p)}(\delta \xi_1)-\delta[\xi_1,\xi_2]=0\,,\quad \forall \xi_i \in {\mathfrak g}\,.
	\end{equation}

	\begin{definition}
		A {\it Lie bialgebra}  is a Lie algebra $(\mathfrak g,[\cdot,\cdot])$, with an additional linear map $\delta:\mathfrak g\to \wedge^2\mathfrak g$ such that:
		\begin{enumerate}
			\item[$(i)$] $\delta$  is a $1$-cocycle for  the cohomology defined by  the representation $\ad^{(2)}:\mathfrak g \times \wedge^2 \mathfrak g \to \wedge^2 \mathfrak g$, 
			\item[$(ii)$] the dual map $[\cdot,\cdot]_{\g^*}:=\delta^{t}: \mathfrak{g}^*\times{\mathfrak g}^* \longrightarrow \mathfrak g^{*}$ of $\delta$  is a Lie bracket on $\mathfrak g^{*}.$
		\end{enumerate}
		We will denote a Lie bialgebra by the triple $(\g,[\cdot,\cdot],\delta)$; however, we sometimes refer to it using the pair of Lie algebras $(\g,\g^*)$. 
	\end{definition}	
	\noindent  Note that $\delta$ is, up to a sign, 
	the algebraic differential associated with the Lie algebra $({\mathfrak g}^*, [\cdot,\cdot]_{\g^*}).$ 
	
	\noindent 	If $(\mathfrak g,[\cdot,\cdot])$ is a Lie algebra and we have any Lie bracket $[\cdot,\cdot]_{\g^*}$ on ${\mathfrak g}^*$ then, using (\ref{cocycle}), one can see $(\g,[\cdot,\cdot],\delta)$ , where $\delta=([\cdot,\cdot]_{\g^*})^t$, is a Lie bialgebra if and only if 
	\begin{equation}\label{bi-alg}
		\begin{array}{rcl}
			[\eta_1,\eta_2]_{\g^*}([\xi_1,\xi_2])&=&[\ad_{\xi_1}^*\eta_1,\eta_2]_{\g^*}(\xi_2) + [\eta_1, \ad_{\xi_1}^*\eta_2]_{\g^*}(\xi_2) \\[8pt]&&-[\ad_{\xi_2}^*\eta_1,\eta_2]_{\g^*}(\xi_1)-[\eta_1, \ad_{\xi_2}^*\eta_2]_{\g^*}(\xi_1),
		\end{array}
	\end{equation}
	for all $(\xi_1,\xi_2,\eta_1,\eta_2)\in {\mathfrak g}\times {\mathfrak g}\times {\mathfrak g}^*\times {\mathfrak g}^*,$ where $\ad^*: {\mathfrak g}\times {\mathfrak g}^*\to {\mathfrak g}^*$ is the coadjoint action. Indeed, if $\delta=([\cdot,\cdot]_{\g^*})^t$, we deduce that 
	\begin{equation}\label{deltaxi}\delta([\xi_1,\xi_2])(\eta_1,\eta_2)=[\eta_1,\eta_2]_{\g^*}([\xi_1,\xi_2]).\end{equation}	
	\noindent	Moreover, using (\ref{ad2}), we have that 
	\begin{equation}\label{adxi}\ad_{\xi_i}^{(2)}(\delta \xi_j)(\eta_1,\eta_2)=\lcf\xi_i,\delta\xi_j\rcf (\eta_1,\eta_2)= [\ad_{\xi_i}^*\eta_1,\eta_2]_{\g^*}(\xi_j) + [\eta_1, \ad_{\xi_i}^*\eta_1]_{\g^*}(\xi_j). \end{equation}
	\noindent Finally, from (\ref{cocycle}), (\ref{deltaxi}) and (\ref{adxi}) we conclude (\ref{bi-alg}). 
	
	\subsection{Coboundary Lie bialgebras}
	A coboundary Lie bialgebra is a specific type of Lie bialgebra characterized by the existence of a coboundary structure. Let $(\mathfrak g,[\cdot,\cdot])$ be a finite-dimensional Lie algebra over the field of real numbers.
    
\begin{definition}
	A $k$-cochain $\delta$, $k\geq 1$, of $\mathfrak g$ with values on $V$, is called a $k$-coboundary if there exists a $(k-1)$-cochain $r$, such that $\delta=\partial r$. It is obvious that any $k$-coboundary is a $k$-cocycle. 
\end{definition}
\begin{definition}
A Lie bialgebra $({\mathfrak g}, [\cdot,\cdot], \delta)$ is called coboundary if $1$-cocycle $\delta$ is the coboundary of an element $r\in \wedge^2 {\mathfrak g}$ , i.e. $\delta=\partial r$, in the cohomology induced by the representation $\ad^{(2)}$. We indicate a coboundary Lie bialgebra by $({\mathfrak g}, [\cdot,\cdot], \delta_r)$. 
\end{definition}

	Let $r\in \wedge^2 {\mathfrak g}$ a $0$-cochain on the cohomology defined by the representation $\ad^{(2)}$ of $(\mathfrak g,[\cdot,\cdot])$. 
The bilinear map $r:{\mathfrak g}^*\times {\mathfrak g}^*\to \R$ induces a linear map $r^{\sharp}: {\mathfrak g}^*\to {\mathfrak g}$ given by 
$$r^{\sharp}(\eta_1)(\eta_2)=r(\eta_1,\eta_2),$$
for $\eta_i\in {\mathfrak g}^*.$ For simplicity we denote the linear map $r^{\sharp}$ by $\rs$. 
	
\noindent		In what follows, we will see the sufficient condition for $\delta_r$ to define a Lie bracket on $\g^*$. Let $r\in \wedge^2\mathfrak g$ be an $r$-matrix for the Lie algebra $(\mathfrak g,[\cdot,\cdot])$, i.e. a solution of the classical Yang-Baxter equation
	$$\lcf r,r\rcf=0\,,$$
for more detail see \cite{Ko}. Now, we consider the $1$-cocycle $\delta_r=\lcf\cdot, r\rcf:{\mathfrak g} \to \wedge^2 {\mathfrak g}$ associated with $r$-matrix $r$ on the cohomology defined by the representation $\ad^{(2)}$ of $(\mathfrak g,[\cdot,\cdot])$ by 
	\begin{equation}\label{deltar}
		\delta_r(\xi)(\eta_1,\eta_2)=\eta_1[\xi,\rs \eta_2]-\eta_2[\xi,\rs \eta_1]\,.
	\end{equation}
	\noindent In fact, using the graded Jacobi identity of $\lcf\cdot,\cdot \rcf $, we have 
	$$\delta_r([\xi_1,\xi_2])=\lcf[\xi_1,\xi_2],r\rcf=\lcf\xi_1,\lcf\xi_2,r\rcf\rcf-\lcf\xi_2,\lcf\xi_1,r\rcf\rcf=\ad^{(2)}_{\xi_1}\delta_r(\xi_2)-\ad^{(2)}_{\xi_2}\delta_r(\xi_1),$$
	for all $\xi_i\in \mathfrak g.$ Then, the corresponding bracket { $\delta_r^t=[\cdot,\cdot]_r:{\mathfrak g}^*\times{\mathfrak g}^*\to {\mathfrak g}^*$} deduced from $\delta_r$  on ${\mathfrak g}^*$ is given by
	\begin{equation}\label{rbracket}
		[\eta_1,\eta_2]_r=\ad^*_{\rs\eta_1}\eta_2-\ad^*_{\rs\eta_2}\eta_1\,,
	\end{equation}
	equivalently,
	\begin{equation}\label{corchete}
		[\eta_1,\eta_2]_r(\xi)=\lcf \xi,r\rcf(\eta_1,\eta_2)=\eta_2([\xi,\rs\eta_1])-\eta_1([\xi,\rs\eta_2])\,,
	\end{equation}
	\noindent for all $\eta_i\in \mathfrak g^*$ and $\xi\in \mathfrak g\,.$ 
	The skew-symmetric bracket $[\cdot,\cdot]_r=\delta_r^t: \mathfrak g^*\times \mathfrak g^*\to \mathfrak g^*$ on $\mathfrak g^*$  {given by (\ref{corchete})} is a Lie bracket. Indeed, from (\ref{corchete}) and Proposition \ref{r}, we deduce that 
	$$
	\begin{array}{rcl}
		[[\eta_i,\eta_j]_r,\eta_k]_r(\xi)&=&[\eta_i,\eta_j]_r([\xi,\rs\eta_k])-\eta_k([\xi, \rs[\eta_i,\eta_j]_r])\\[5pt]
		&=&{ \eta_i( [ [\xi,\rs\,\eta_k],\rs\eta_j] )} -\eta_j([[\xi,\rs\eta_k],\rs\eta_i]){  -\eta_k([\xi,[\rs\eta_i,\rs\eta_j]])}\,,
	\end{array}
	$$
	for all $i,j,k\in \{1,2,3\}$ and $\xi\in {\mathfrak g}$. 	Now, using the Jacobi identity of the Lie bracket $[\cdot,\cdot]$, we deduce that $[\cdot,\cdot]_r$ is a Lie bracket.

	The following result characterizes the solutions of  Yang-Baxter equations in terms of the linear map $\rs$, for more detail see (\cite{Ko}). 
	
	\begin{proposition}\label{r}
		Let $({\mathfrak g},[\cdot,\cdot])$ be a Lie algebra and $r$ an element of $\wedge^2 {\mathfrak g} $. { Then, the following are equivalent:}
		\begin{enumerate}
			\item $r$ is a solution of the classical Yang-Baxter equations.
			\item
			$\rs([\eta_1,\eta_2]_r)=-[\rs\eta_1, \rs\eta_2],$
			for all $\eta_i\in {\mathfrak g}^*,$ where $[\cdot,\cdot]_r$ is the bracket defined by (\ref{corchete}). 
		\end{enumerate}
	\end{proposition}
	
	\begin{proof}

		The condition $\lcf r,r\rcf=0$ is equivalent with the relation 
		$$ \eta_1([\rs\eta_2, \rs\eta_3])-\eta_2([\rs\eta_1, \rs\eta_3]) + \eta_3([\rs\eta_1, \rs\eta_2])=0,$$
		for all $\eta_i\in {\mathfrak g}^*.$ On the other hand, using the anti-symmetric property of $r$ and (\ref{corchete}), we have 
		$$\eta_3(\rs[\eta_1,\eta_2]_r)=-[\eta_1,\eta_2]_r(\rs\eta_3)= \eta_2([\rs\eta_1, \rs\eta_3]) - \eta_1([\rs\eta_2, \rs\eta_3]).$$
		From the two previous relations, we deduce the proposition.	
	\end{proof}

	\begin{corollary}
		Let $({\mathfrak g}, [\cdot,\cdot])$ be a Lie algebra and $r\in \wedge^2{\mathfrak g}$ a solution of the classical Yang-Baxter equation. Then 
		$({\mathfrak g}, [\cdot,\cdot], \delta_r)$ is a coboundary Lie bialgebra, where $\delta_r=([\cdot,\cdot]_r)^t$ and $[\cdot,\cdot]_r$ is the bracket given by (\ref{corchete}).
	\end{corollary}
	\section{Deformation of Lie algebras and Lie bialgebras}\label{sec3}
	In this section, we explore three types of deformation of the corresponding $1$-cocycle associated with a Lie bialgebra $({\mathfrak g}, [\cdot,\cdot], \delta)$, utilizing an (almost) Nijenhuis structure on the Lie algebra $({\mathfrak g}, [\cdot,\cdot])$. Under certain conditions, these deformations yield new $1$-cocycles.

	\subsection{One-cocycle on a deformed Lie algebra}
	
	Consider a Lie algebra $(\mathfrak g,[\cdot,\cdot])$ and suppose we have a linear operator  $n:\mathfrak g \to \mathfrak g$ on it. The linear map $n$ defines a skew symmetric deformed bracket on $\g$ by
	\begin{equation}\label{defn}
		[\xi_1,\xi_2]_n=[n\xi_1,\xi_2]+[\xi_1,n\xi_2]-n[\xi_1,\xi_2] \quad \xi_i\in  \mathfrak{g}\,,
	\end{equation}
	and the Nijenhuis torsion of $n$ with respect to the Lie bracket $[\cdot,\cdot]$ is given by
	\begin{equation}\label{bracket--n}
		\lcf n,n\rcf(\xi_1,\xi_2)=n[\xi_1,\xi_2]_n-[n\xi_1,n\xi_2]\quad \xi_i\in \mathfrak g\,.
	\end{equation}
	The linear map  $n$ is {\it almost Nijenhuis} if $\lcf n,n\rcf$ is a $2$-cocyle, in the cohomology complex induced by the adjoint representation  of $(\mathfrak g,[\cdot,\cdot]$), and it is {\it Nijenhuis} if $\lcf n,n\rcf=0$.

	\begin{proposition}\label{Lie}
		Let $(\mathfrak g,[\cdot,\cdot])$ be a Lie algebra and  $n:\mathfrak g \to \mathfrak g$ a linear operator on $\mathfrak g.$ Then,  $n$ is almost Nijenhuis 
		if and only if $[\cdot,\cdot]_n$ is a Lie bracket. In such a case  $[\cdot,\cdot]$ and $[\cdot,\cdot]_n$ are compatible Lie brackets, i.e. $[\cdot,\cdot] + [\cdot,\cdot]_n$ is a new Lie bracket.
	\end{proposition}
	\begin{proof}
		
		Using (\ref{defn}), (\ref{bracket--n}) and the Jacobi identity, after a straightforward computation, we obtain 
		$$\begin{array}{rcl}
			\partial\lcf n,n\rcf(\xi_1,\xi_2, \xi_3)&=& [\xi_1, \lcf n,n\rcf(\xi_2,\xi_3)]- [\xi_2, \lcf n,n\rcf(\xi_1,\xi_3)]+ [\xi_3, \lcf n,n\rcf(\xi_1,\xi_2)]\\[8pt]&& - \lcf n,n\rcf([\xi_1,\xi_2], \xi_3) + \lcf n,n\rcf([\xi_1,\xi_3], \xi_2)-\lcf n,n\rcf([\xi_2,\xi_3], \xi_1)\\[8pt]&=& { [[\xi_1,\xi_2]_n\,,\xi_3]_n + [[\xi_3,\xi_1]_n\,,\xi_2]_n+ [[\xi_2,\xi_3]_n\,,\xi_1]_n,}
		\end{array}$$
		for all $\xi_i\in {\mathfrak g},$ where  $\partial$ denotes the  differential induced by the adjoint representation of the Lie bracket $[\cdot,\cdot]$. If $n$ is an almost Nijenhuis operator then using that $[\cdot,\cdot]_n$ and $[\cdot,\cdot]$ satisfy the Jacobi identity for $(\xi_1,\xi_2,\xi_3)$, $(n\xi_1,\xi_2,\xi_3)$, $(\xi_1,n\xi_2,\xi_3)$
		and $(\xi_1,\xi_2,n\xi_3)$, we deduce that $[\cdot,\cdot] + [\cdot,\cdot]_n$ is a Lie bracket.
	\end{proof}
	\begin{remark}
		Note that the
		Nijenhuis torsion measures the failure of the linear map n to define a Lie algebra homomorphism $n:(\mathfrak g, [\cdot,\cdot]_n) \to (\mathfrak g, [\cdot,\cdot])$.
	\end{remark}
	We denote the adjoint representation of the Lie algebra $(\g,[\cdot,\cdot]_n)$ by ${\ad}^n:\g \times \g\to \g$. It is defined by
	\begin{equation}\label{ad-n}
		{\ad}^n_{\xi}=[\ad_{\xi},n]+\ad_{n\xi}\,,\quad \xi\in\g\,,
	\end{equation}
	
	\noindent where $[\ad_{\xi},n]=\ad_{\xi}n-n\ad_{\xi}$. For the simplicity we denote this commutator by 
\begin{equation}
	\underline{\ad}_{\,\xi}:=[\ad_{\xi},n]\,.
\end{equation}
	 The coadjoint representation of the deformed Lie bracket $[\cdot,\cdot]_n$ is then defined by 
	\begin{equation}\label{ad*-n}
		({\ad}^n)^*_{\xi}=[{}^tn,\ad^*_{\xi}]+\ad^*_{n\xi}\,,\quad \xi\in\g\,,
	\end{equation}
	\noindent where
		\begin{equation} {\ads}^*_{\,\xi}:= [{ }^tn,\ad^*_{\xi}]={ }^tn\ad^*_{\xi}-\ad^*_{\xi} { }^t n\,. \end{equation}
	
	\begin{definition}
		A linear operator $n$ on the Lie algebra $\g$ and the adjoint representation of $\g$ are compatible if $C(n,\ad)\equiv 0$, where
		\begin{equation}\label{com1}	
			C(n,\ad)(\xi_1,\xi_2)=[[\ad_{\xi_1},n],[\ad_{\xi_2},n]]+[\ad_{{n[\xi_1,\xi_2]}},n]-[\ad_{\xi_1},[\ad_{n\xi_2},n]-[[\ad_{n\xi_1},n],\ad_{\xi_2}]\,.
		\end{equation}
	\end{definition}
	\noindent 	It is easy to see $C(n,\ad)\equiv0$ if and only if $[\cdot,\cdot]_n$ is a Lie bracket on $\g$.	When $n$ is a Nijenhuis operator, the deformed bracket is a Lie bracket and $\ad^n$ is a representation of the Lie algebra $(\g,[\cdot,\cdot]_n)$ and so $C(n,\ad)\equiv 0$.
 
 	\medskip 
We will adopt the following notation:
	
\noindent For any linear map $\phi:\g \to \g$ on a Lie algebra $\g$, we define the operator $\iota_{\phi}:\wedge^k\mathfrak{g}\to \wedge^k\mathfrak{g}$,  by 
	\begin{equation}\label{phi}
		(\iota_{\phi}P)(\eta_1,\dots \eta_k)=\sum_{i=1}^kP(\eta_1,\dots,\phi(\eta_i),\dots ,\eta_k)\,,	\end{equation}
	for all {$\eta_i\in \mathfrak{g^*}$} and $P\in \wedge^k\mathfrak{g}.$
	\begin{lemma}
	The assignment $\phi \mapsto\iota_{\phi}$
defines an anti-automorphism: $\gl(g) \to \gl(\wedge^k\g)$, that is
		\begin{equation}\label{phii}
			\iota_{[\phi_1,\phi_2]}=-[\iota_{\phi_1},\iota_{\phi_2}]\,.
		\end{equation}
	\end{lemma} 
	\begin{proof}
	The proof is straightforward by using 	\begin{equation}\label{phiii}
			\iota_{\phi_1\phi_2}P(\eta_1,\dots \eta_k)=\iota_{\phi_2}\iota_{\phi_1}P(\eta_1,\dots \eta_k)-\sum\limits_{i,j=1,i\neq j}P(\eta_1,\dots,\phi_1\eta_i,\dots,\phi_2\eta_j,\dots,\eta_k)\,.	\mbox{\qedhere}
		\end{equation}
\end{proof}
	Now, let $(\g,[\cdot,\cdot],\delta)$ be a Lie bialgebra and $n$ a Nijenhuis operator on $(\g,[\cdot,\cdot])$. We consider $\delta$ as a cochain in the deformed cohomology, the cohomology induced by $(\ad^n)^{(2)}$ on the Lie algebra $(\g,[\cdot,\cdot]_n)$.  In this context, to denote the specific cohomology being considered, we write $\dsn:\g\to\wedge^2\g$.	By definition, $\dsn$ is a $1$-cocycle in the deformed cohomology if and only if $\partial\dsn\equiv 0$, i.e.
	\begin{equation}\label{deln}
		\partial\dsn(\xi_1,\xi_2)=\iota_{[^tn,\ad^*_{\xi_1}]} \delta(\xi_2)-\iota_{[^tn,\ad^*_{\xi_2}]} \delta(\xi_1)+\ad^{(2)}_{n\xi_1}\delta(\xi_2)-\ad^{(2)}_{n\xi_2}\delta(\xi_1)-\delta[\xi_1,\xi_2]_n=0\,.
	\end{equation}
	In this case, $\dsn$ defines a new Lie bialgebra $(\g,[\cdot,\cdot]_n, \delta^n)$. Note that $\delta$ and $\dsn$ satisfy the same defining conditions, even though for different Lie algebras.
	
	\subsection{Deformation of a one-cocycle on a Lie algebra: almost NL bialgebras}
	
	Let $(\g,[\cdot,\cdot], \delta)$ be a Lie bialgebra and $n:\g \to \g$ an (almost) Nijenhuis operator on $(\g,[\cdot,\cdot])$. If $^tn:\g^* \to \g^*$ is an (almost) Nijenhuis structure on $\g^*$ then we have a new Lie bracket on the dual Lie algebra $\g^*$ defined by 
	\begin{equation}\label{nt}
		[\eta_1,\eta_2]^{{}^tn} = [^tn\eta_1,\eta_2]_{\g^*} +[\eta_1, {}^tn\eta_2]_{\g^*}-{}^tn[\eta_1,\eta_2]_{\g^*},\mbox{ for all }\eta_i\in \mathfrak{g}^*.
	\end{equation}
	
	\noindent Therefore, we have a new pair of Lie algebras $((\mathfrak g,[\cdot,\cdot]), (\mathfrak g^*,[\cdot,\cdot]^{^tn}))$. The goal is to look for the condition that this defines again a Lie bialgebra. We denote $\dn$ the deformed $1$-cocycle by $^tn$ associated with this Lie bialgebra. Note that for the sake of simplicity, we use the notation $[\cdot,\cdot]^{^tn}$ instead of $[\cdot,\cdot]^{^tn}_{\g^*}$. 
 
	\begin{proposition}\label{prop-equiv1}
		Let  $(\mathfrak g,[\cdot,\cdot]), \delta)$ be a Lie bialgebra and  $n:\mathfrak g \to \mathfrak g$ (respectively, $^tn:\mathfrak g^* \to \mathfrak g^*$) an (almost)  Nijenhuis operator on the Lie algebra $(\mathfrak g,[\cdot,\cdot])$ (respectively, $(\mathfrak g^*,[\cdot,\cdot]_{\g^*})$). The pair $((\mathfrak g,[\cdot,\cdot]), (\mathfrak g^*,[\cdot,\cdot]^{^tn}))$ forms a Lie bialgebra if and only if 
		$\iota_{{}^tn}\circ \d^*-\d^*\circ n:\mathfrak g\to \wedge^2\mathfrak g$ is a $1$-cocycle for the cohomology induced by representation  $\ad^{(2)}:  \g\times \wedge^2\g\to \wedge^2\g$ of $(\g, [\cdot,\cdot])$, where $\d^*$ is the algebraic differential induced by  the Lie bracket $[\cdot,\cdot]_{\g^*}$. 
		
	\end{proposition}
	\begin{proof}
		
		First, using Proposition \ref{Lie}, $(\mathfrak g,[\cdot,\cdot]_n)$ and  $(\mathfrak g^*,[\cdot,\cdot]^{^tn})$ are  Lie algebras. Since  $\d^*\xi(\eta_1,\eta_2)=[\eta_1,\eta_2]_{\g^*}(\xi)$ for all $\xi\in {\mathfrak g}$ and $\eta_i\in {\mathfrak g}^*,$ 
		\begin{equation}\label{d1}
			([\cdot,\cdot]^{^tn})^t(\xi)(\eta_1,\eta_2)=[\eta_1,\eta_2]^{^tn}(\xi)= \iota_{^tn}(\d^*\xi)(\eta_1,\eta_2)-\d^*(n\xi)(\eta_1,\eta_2)\,,\end{equation}
where $\dn=([\cdot,\cdot]^{^tn})^t$. Therefore,  $\dn$ is a $1$-cocycle if and only if $\iota_{^tn}\circ \d^*-\d^*\circ n: \mathfrak g\to \wedge^2\mathfrak g$ is a $1$-cocycle for the cohomology induced by $\ad^{(2)}$of $({\mathfrak g}, [\cdot,\cdot])$.	
	\end{proof}
The following Corollary gives the relation between the deformed $1$-cocycle $\dn$ for the bialgebra $(\mathfrak g,[\cdot,\cdot], \dn)$ and $1$-cocycle $\delta$ for the bialgebra $(\mathfrak g,[\cdot,\cdot],\delta)$.
	\begin{corollary}
	Let $(\g, [\cdot, \cdot], \delta, n)$ be as in Proposition \ref{prop-equiv1}, then
		\begin{equation}\label{Dn}
			\dn=\iota_{^tn}\delta-\delta \circ n\,.\end{equation}  
	\end{corollary}	
	\smallskip
	\begin{corollary}\label{adadn}
		For a $1$-cocycle $\delta$ and an (almost) Nijenhuis structure $n$ on $\g$, we have  \begin{equation}\label{adnnad}
			\iota_{\ads_{\,\xi}^*}\delta=\iota_{[^tn,\ad^*_{\xi}]}\delta=\ad^{(2)}_{\xi}\iota_{^tn}\delta-\iota_{^tn}\ad^{(2)}_{\xi}\delta\,.	\end{equation}
	\end{corollary}
	\begin{proof}
By (\ref{phi}), for the adjoint representation $\ad_{\xi}$ on the Lie algebra $(\mathfrak g,[\cdot,\cdot])$, $\ad^{(2)}_\xi=\iota_{\ad^*_\xi}\,$; and then by (\ref{phii}) we conclude (\ref{adnnad}).
	\end{proof}
	\begin{proposition}\label{prop-equiv}
		Let $(\mathfrak g,[\cdot,\cdot], \delta)$ be a Lie bialgebra, and let $n:\mathfrak g \to \mathfrak g$ be an (almost) Nijenhuis operator on $(\mathfrak g,[\cdot,\cdot])$. Then $\dn$ is a $1$-cocycle if and only if $\delta$ is a $1$-cocycle in the the deformed cohomology defined by $[\cdot,\cdot]_n$\,.
		
	\end{proposition}
	\begin{proof}	
		
		By definition, $\dn$ is a $1$-cocycle if and only if $\partial{\dn}\equiv 0$, where by (\ref{Dn}) 
		\begin{equation}\label{ndel}
			\partial{\dn}(\xi_1,\xi_2)=\ad^{(2)}_{\xi_1}\iota_{^tn}\delta(\xi_2)-\ad^{(2)}_{\xi_2}\iota_{^tn}\delta(\xi_1)-\ad^{(2)}_{\xi_1}\delta(n\xi_2)+\ad^{(2)}_{\xi_2}\delta(n\xi_1)-\dn([\xi_1,\xi_2])\,.
		\end{equation}
		On the other hand, by (\ref{Dn}) and since $\delta$ is a $1$-cocycle
		$$\dn([\xi_1,\xi_2])=\iota_{^tn} \ad^{(2)}_{\xi_1}\delta(\xi_2)-\iota_{^tn} \ad^{(2)}_{\xi_2}\delta(\xi_1)-\delta(n[\xi_1,\xi_2])\,.$$
		
		\noindent 	Then, using (\ref{adnnad}), we get 
		\begin{equation}\label{2}
			\partial{\dn}(\xi_1,\xi_2)=\iota_{[^tn,\ad^*_{\xi_1}]}\delta(\xi_2)- \iota_{[^tn,\ad^*_{\xi_2}]}\delta(\xi_1)-\ad^{(2)}_{\xi_1}\delta(n\xi_2)+\ad^{(2)}_{\xi_2}\delta(n\xi_1)+\delta(n[\xi_1,\xi_2])\,.
		\end{equation}
		\noindent Finally, by substituting
		$$-\ad^{(2)}_{\xi_1}\delta(n\xi_2)+\ad^{(2)}_{\xi_2}\delta(n\xi_1)=-\delta[\xi_1,n\xi_2]-\ad^{(2)}_{n\xi_2}\delta(\xi_1)-\delta[n\xi_1,\xi_2]+\ad^{(2)}_{n\xi_1}\delta(\xi_2)\,,$$
		and (\ref{defn}) in (\ref{2}) we get (\ref{deln}) and hence $$\partial \dn\equiv 0 \iff \partial \dsn\equiv 0\,. \mbox{\qedhere}$$
\end{proof}

	\begin{definition}
		Let $(\g,[\cdot,\cdot], \delta)$ be a Lie bialgebra and  $n$ be an (almost) Nijenhuis operator on the Lie algebra $(\g,[\cdot,\cdot])$. We say  $(\g,[\cdot,\cdot], \delta,n)$ is an {\it almost NL bialgebra} if
		\begin{enumerate}
			\item[(i)]		 $\delta$ is also a $1$-cocycle in the deformed cohomology, i.e. $\partial\dsn\equiv 0$. 
			
			\item[(ii)]	$C(^tn, \ad^*)\equiv 0$ where $\ad^*:\g^*\times \g^*\to\g^*$ is the adjoint representation of $(\g^*,[\cdot,\cdot]_{\g^*})$.
		\end{enumerate}	
	\end{definition}
	\noindent 	The condition $(i)$ shows that $(\mathfrak g,[\cdot,\cdot]_n,\delta)$ is a Lie bialgebra. Furthermore, according to the Proposition \ref{prop-equiv}, the deformation $\dn$ of $\delta$ by $^tn$ is also a $1$-cocycle, i.e. $\partial \delta_{^tn}\equiv 0$. The condition $(ii)$ implies $[\cdot,\cdot]^{^tn}$ is a Lie bracket on $\g^*$ and so
	$(\mathfrak g,[\cdot,\cdot], \delta_{^tn})$ is also a Lie bialgebra.
	\smallskip

	\subsection{Double Deformation: (weak) NL bialgebras} 	
		\noindent Let $(\g,[\cdot,\cdot],\delta,n)$ be an almost NL bialgebra. We can deform the $1$-cocycle $\dsn$ by means of $^tn$, or take $\delta_{^tn}$ on the deformed cohomology defined by $n$. The result is the same, a double deformed $1$-cochain $\dsn_{\,^tn}:\g\to  \wedge ^2\g $ by means of $(n,{}^tn)$ in the deformed cohomology induced by $(\ad^n)^{(2)}:\g\times \wedge ^2\g\to  \wedge ^2\g $ on the Lie algebra $(\g,[\cdot,\cdot]_n)$, see the diagram below.  
\begin{figure}[H]
	\subfloat{
    \xymatrix@R-5mm @C-7mm{
        & \delta:((\g,[\cdot,\cdot]),(\g^*,[\cdot,\cdot]_{\g^*})) \ar@{.>}[dl]_{n} \ar@{.>}[dr]^{^tn} & \cr
        \dsn: ((\g,[\cdot,\cdot]_n),(\g^*,[\cdot,\cdot]_{\g^*})) \ar@{.>}[dr]_{^tn} & & \dn: ((\g,[\cdot,\cdot]),(\g^*,[\cdot,\cdot]^{^tn})) \ar@{.>}[dl]^{~~~~n} \cr
        & \dsn_{^tn}: ((\g,[\cdot,\cdot]_n),(\g^*,[\cdot,\cdot]^{^tn})) &
    }
}
	\caption{Pair of Lie algebras defined by the corresponding $1$-cocycles of Lie bialgebras.}	\end{figure}
     In the following, we discuss the conditions under which the double deformed $1$-cochain $\dsn_{\,^tn}$ constitutes a $1$-cocycle in the deformed cohomology.  This analysis will ultimately lead us to the formulation of  (weak) NL bialgebras.
		\begin{lemma}\label{lemma.rel1}
		For a Nijenhuis operator $n$, and a linear map $\delta:\g\to  \wedge ^2\g $ on $(\g,[\cdot,\cdot])$, we have
		\begin{equation}\label{rel1}
		\ads^*_{\,n\xi}=\ads_{\,\xi}^*\circ {}^tn.
		\end{equation}
			\end{lemma}
		\begin{proof}
			Since $n$ is a Nijenhuis operator, the commutator of endomorphism  $[\ad_{n\xi},n]=n[\ad_\xi,n]$ on $\g$ holds and so $[{}^tn, \ad^*_{n\xi}]=[{}^tn,\ad^*_\xi]\circ {}^tn$ on $\g^*$. Hence, we derive (\ref{rel1}).
		\end{proof}

	\begin{theorem}\label{nnt}
		Let $(\g,[\cdot,\cdot],\delta,n)$ be an almost NL bialgebra. The double deformed $1$-cochain $\dsn_{\,^tn}$ is a $1$-cocycle in the deformed cohomology induced by $(\ad^n)^{(2)}:\g\times \wedge ^2\g\to  \wedge ^2\g $ on the Lie algebra $(\g,[\cdot,\cdot]_n)$ if and only if 
		\begin{equation}\label{com3}
			\iota _{^tn\circ \ads^*_{\,\xi_1}}\delta(\xi_2)-
			\iota _{^tn\circ\ads^*_{\,\xi_2}}\delta(\xi_1)=\iota_{\ads^*_{\,\xi_1}}\delta(n\xi_2)-\iota_{\ads^*_{\,\xi_2}}\delta(n\xi_1)\,.
		\end{equation}
	\end{theorem}
\begin{proof}
	By definition, $\dsn_{^tn}$ is a $1$-cocycle if 
	\begin{equation}\label{delnnt}
		(\adn_{\,\xi_1})^{(2)}\delta_{^tn}(\xi_2)-({\adn_{\,\xi_2}})^{(2)}\delta_{^tn}(\,\xi_1)=\delta_{^tn}[\xi_1,\xi_2]_n\,.
	\end{equation}
	For the left hand side of (\ref{delnnt}), by (\ref{ad-n}), (\ref{adnnad}) and $\partial \delta\equiv 0$, we get
	\begin{equation}\label{left}
		\iota_{\ads^*_{\,\xi_1}}(\iota_{^tn}\delta(\xi_2))-\iota_{\ads^*_{\,\xi_2}}(\iota_{^tn}\delta(\xi_1))+\adt_{n\xi_1}\iota_{^tn}\delta(\xi_2)-\adt_{n\xi_2}\iota_{^tn}\delta(\xi_2)-\iota_{\ads^*_{\,\xi_1}}\delta(n\xi_2)+\iota_{\ads^*_{\,\xi_2}}\delta(n\xi_1)-\delta[n\xi_1,n\xi_2]\,.
	\end{equation}
	
	\noindent For the right hand side of (\ref{delnnt}), by (\ref{defn}),  (\ref{delnnt}) and $\partial\dsn\equiv 0$, we get
	\begin{equation}\label{right}
		\iota_{^tn}(\iota_{\ads^*_{\,\xi_1}}\delta(\xi_2))-\iota_{^tn}(\iota_{\ads^*_{\,\xi_2}}\delta(\xi_1))+\iota_{^tn} \adt_{n\xi_1}\delta(\xi_2)-\iota_{^tn}\adt_{n\xi_2}\delta(n\xi_1)-\delta[n\xi_1,n\xi_2]\,.
	\end{equation}
	\noindent By (\ref{phii}), we have
	\begin{equation}\label{rel2}
		\iota_{[^tn,\ads^*_{\,\xi_1}]}\delta(\xi_2)(\eta_1,\eta_2)=-[\iota_{^tn},\iota_{\ads^*_{\,\xi_1}}]\delta(\xi_2)(\eta_1,\eta_2)\,,\quad \forall\xi_1,\xi_2 \in \g\,,\eta_1,\eta_2\in \g^*.
	\end{equation}
	
	\noindent Finally, by substituting (\ref{left}) and (\ref{right}) in (\ref{delnnt}) and then using (\ref{rel1}) and (\ref{rel2}) we get (\ref{com3}).
\end{proof}

	\begin{definition}
		An almost NL bialgebra $(\g,[\cdot,\cdot],\delta,n)$  is a {\it weak NL bialgebra} if (\ref{com3}) is satisfied. 
	\end{definition}
	\begin{definition}
		Let $(\g,[\cdot,\cdot],\delta)$ be a Lie bialgebra and let $n$ be a Nijenhuis structure on the Lie algebra $(\g,[\cdot,\cdot])$. We define the concomitant $C(\delta,n)$ as a $(2,1)$-tensor field on $\g$ and we interpret the expression in (\ref{com3}) as the condition for the vanishing of this tensor field.
			\end{definition}
		\begin{comment}
	By	 (\ref{bi-alg}), the concomitant $C(\delta,n)$ can be written as
	\[
	\begin{array}{rcl}
		&&	C(\delta,n)(\eta_1,\eta_2)([\xi_1,\xi_2])=\\[5pt]
		&&[{}^tn\ads^*_{\,\xi_1}\eta_1,\eta_2]_{\g^*}(\xi_2)+	[\eta_1,{}^tn\ads^*_{\,\xi_1}\eta_2]_{\g^*}(\xi_2)-[{}^tn\ads^*_{\,\xi_2}\eta_1,\eta_2]_{\g^*}(\xi_1)-	[\eta_1,{}^tn\ads^*_{\,\xi_2}\eta_2]_{\g^*}(\xi_1)	-\\[5pt]
		&&{}^tn[\ads^*_{\,\xi_1}\eta_1,\eta_2]_{\g^*}(\xi_2)-{}^tn[\eta_1,\ads^*_{\,\xi_1}\eta_2]_{\g^*}(\xi_2)+{}^tn[\ads^*_{\,\xi_2}\eta_1,\eta_2]_{\g^*}(\xi_1)+{}^tn[\eta_1,\ads^*_{\,\xi_2}\eta_2]_{\g^*}(\xi_1)\,.
	\end{array}
	\]	
When the concomitant vanishes, we say that $n$ is $\ad^{(2)}$-equivariant. If we have $$[\iota_{\ad^*_{\xi_1}}\lcf {}^tn,{}^tn \rcf (\eta_1,\eta_2)](\xi_2)-[\iota_{\ad^*_{\xi_2}}\lcf {}^tn,{}^tn \rcf (\eta_1,\eta_2)](\xi_2)=0\,,$$
then $C(\delta,n)=0$ is equivalent to
\[
\begin{array}{rcl}
\ads^*_{\xi_1}\dn(\xi_2)(\eta_1,\eta_2)-\circlearrowleft &=&\delta(\xi_2)((n^t\eta_1,\ads^*_{\xi_1}\eta_2)+(\ads^*_{\xi_1}\eta_1, {}^tn\eta_2))-\circlearrowleft
\end{array}
\]
  where by $\circlearrowleft$ we mean the permutation of $\xi_1$ and $\xi_2\,$.
	\end{comment}

	\begin{definition}
Let $n:\g\to\g$ be a Nijenhuis operator on a Lie algebra $(\g,[\cdot,\cdot])$ and $(\g,[\cdot,\cdot],\delta)$ be a Lie bialgebra. A tuple $(\g,[\cdot,\cdot],\delta,n)$ is an {\it NL bialgebra}  if satisfies the following conditions
		\begin{enumerate}
			\item [(1)] $\partial\delta^n\equiv 0$ 
			\item [(2)] $\lcf {}^tn,{}^tn\rcf \equiv 0$
			\item [(3)] $C(\delta,n)\equiv0$
		\end{enumerate}	
    \end{definition}
	
 We remark that for a weak NL bialgebra the condition $\lcf{}^tn,{}^tn\rcf \equiv 0$ is reduced to $C(^tn,\ad^*)\equiv 0$, equivalently reduced to $\partial \lcf{} ^tn,{}^tn\rcf \equiv 0$ with respect to the cohomology induced by the adjoint representation of the Lie algebra $(\g^*,[\cdot,\cdot]_{\g^*})$.
	\section{Solutions of the classical Yang-Baxter equation: coboundary NL bialgebras}\label{sec4}
In this section, we establish fundamental results concerning coboundary Lie bialgebras, focusing on the conditions under which a coboundary Lie bialgebra, in conjunction with a Nijenhuis operator, can be used to construct an NL bialgebra. 
	
	\noindent Suppose that  we have a solution of the classical Yang-Baxter equation $r\in \wedge^2\mathfrak g$ and  a  linear map $n: \mathfrak g\to  \mathfrak g$. Then one may take the bilinear map $nr:\mathfrak g^*\times \mathfrak g^*\to {\mathbb R}$ defined by the linear map $(nr)^{\sharp}:\g^*\to \g$ such that
	$n\rs:=(nr)^{\sharp}=n\circ \rs.$ 
	\begin{proposition}\label{nr}
		Let $r\in \wedge^2{\mathfrak g}$ be a solution of the classical Yang-Baxter equation of the Lie algebra $({\mathfrak g},[\cdot,\cdot])$ and $n: \mathfrak g\to  \mathfrak g$ be a linear map such that
		\begin{enumerate}
			\item [(i)] $n\circ \rs=\rs\circ {}^tn\,,$
			\item [(ii)] $C(r,n)(\eta_1,\eta_2)=\ads^*_{\,\rs\eta_2}\eta_1-\ads^*_{\,\rs\eta_1}\eta_2=0,\quad \forall \eta_i\in\g^*\,.$ 
		\end{enumerate}
		\noindent Here $C(r,n)$ is the concomitant of $r$ and $n$, i.e. the skew-symmetric bilinear map $C(r,n):{\mathfrak g}^*\times {\mathfrak g}^*\to {\mathfrak g}^*$ which is given by 
		\begin{equation}\label{conr}
			C(r,n)(\eta_1,\eta_2)=[{}^tn,\ad^*_{\rs\eta_1}]\eta_2-[{}^tn,\ad^*_{\rs\eta_2}]\eta_1\,.
		\end{equation}
		\noindent	Then $n\rs\in \wedge^2{\mathfrak g}$ is a solution of the classical Yang-Baxter equation on $\mathfrak g$ if and only if 
		$$\lcf n,n\rcf(\rs\eta_1,\rs\eta_2)=0,\quad \mbox{for all}\quad  \eta_i\in \mathfrak g^*\,.$$  
	\end{proposition}
	\begin{proof}
		The condition $(i)$ implies the bilinear map $nr$ is skew symmetric. By (\ref{rbracket}), we see that the condition  $(ii)$ implies 
		$$[\eta_1,\eta_2]_{nr}=[\eta_1,\eta_2]_r^{^tn}\,,$$
		where $[\cdot,\cdot]_{nr}$ is the bracket on ${\mathfrak  g}^*$ defined by 
		$$[\eta_1,\eta_2]_{nr}(\xi)=\lcf \xi,nr\rcf(\eta_1,\eta_2)\,,$$ and 
		\begin{equation}\label{nr*}
			[\eta_1,\eta_2]_{r}^{^tn}= [{}^tn\eta_1,\eta_2]_{r}+ [\eta_1,{}^tn\eta_2]_{r}-{}^tn[\eta_1,\eta_2]_{r},\quad \forall \eta_i\in \g^*\,.
		\end{equation}
		Now, from Proposition \ref{r}, it is sufficient to prove  
		\begin{equation}\label{nr-d}
			n\rs[\eta_1,\eta_2]_{nr}=-[n\rs\eta_1, n\rs\eta_2]+ \lcf n,n\rcf(\rs\eta_1,\rs \eta_2).
		\end{equation}
		In fact, using the conditions $(i)$ and $(ii)$, we have
		$$ n\rs[\eta_1,\eta_2]_{nr}=n\rs([{}^tn\eta_1,\eta_2]_{r}+ [\eta_1,{}^tn\eta_2]_{r}-{}^tn[\eta_1,\eta_2]_{r})\,,$$
		and then using Proposition \ref{r} and condition $(i)$, 
		$$n\rs[\eta_1,\eta_2]_{nr}=-n[n\rs\eta_1,\rs\eta_2]- n[\rs\eta_1,n\rs\eta_2]+n^2[\rs\eta_1,\rs\eta_2].$$
		\noindent Finally (\ref{nr-d}) is deduced since 
		$$
		\lcf n, n\rcf (\rs\eta_1,\rs\eta_2)= [n\rs\eta_1,n\rs\eta_2]+n^2[\rs\eta_1,\rs\eta_2]\\[8pt]-n[n\rs\eta_1,\rs\eta_2]-n[\rs\eta_1,n\rs\eta_2]\,.\mbox{\qedhere}$$	\end{proof}
	\begin{corollary}\label{conn}
		If  $r\in \wedge^2\mathfrak g$ is a non-degenerate solution of the classical Yang-Baxter equation and   $n: \mathfrak g\to  \mathfrak g$ is a  linear map such that $n\circ \rs=\rs\circ {}^tn$ and $C(r,n)=0,$ then $nr\in \wedge^2\mathfrak g$ is a solution of the classical Yang-Baxter equation on $\mathfrak g$ if and only if $n$ is Nijenhuis. 
	\end{corollary}
	\begin{lemma}
		The concomitant (\ref{conr}) in terms of the adjoint representation $\adn$ on the Lie algebra $(\g,[\cdot,\cdot]_n)$ reads
		\begin{equation}\label{conn2}
			[C(r,n)(\eta_1,\eta_2)](\xi)=[(\adn_{\xi})^*\eta_1-{}^tn\ad^*_{\xi}\eta_1](\rs\eta_2)-[(\adn_{\xi})^*\eta_2-{}^tn\ad^*_{\xi}\eta_2](\rs\eta_1)\,.
		\end{equation}
	\end{lemma}
	\begin{proof}
		Applying $\xi$ on the concomitant (\ref{conr}), we get
		$$
		[C(r,n)(\eta_1,\eta_2)](\xi)=(\ad^*_{\rs\eta_2}{}^tn\eta_1)(\xi)-(\ad^*_{\rs\eta_2}\eta_1)(n\xi)-(\ad^*_{\rs\eta_1}{}^tn\eta_2)(\xi)-(\ad^*_{\rs\eta_1}\eta_2)(n\xi)\,.
		$$	
		Using $\left\langle\ad^*_{\xi_1}\eta,\xi_2\right\rangle=-\left\langle\ad^*_{\xi_2}\eta,\xi_1\right\rangle$, it turns
		\begin{equation}\label{com5}
			[C(r,n)(\eta_1,\eta_2)](\xi)=[\ad^*_{\xi}{}^tn\eta_1-\ad^*_{n\xi}\eta_1](\rs\eta_2)-[\ad^*_{\xi}{}^tn\eta_2-\ad^*_{n\xi}\eta_2](\rs\eta_1)\,,
		\end{equation}
		and so by (\ref{ad*-n}), we conclude (\ref{conn2}).
			\end{proof}
	
	\begin{remark}
		Using $\left\langle \ad^*_{\xi_1}\eta,\xi_2\right\rangle =-\left\langle\eta,\ad_{\xi_1}\xi_2\right\rangle$, (\ref{conn2}) turns to
		$$
		[C(r,n)(\eta_1,\eta_2)](\xi)=\eta_1([\xi,n\rs\eta_2]-[\xi,\rs\eta_2]_n)-\eta_2([\xi,n\rs\eta_1]-[\xi,\rs\eta_1]_n)\,,	
		$$
		and (\ref{com5}) turns to
		\begin{equation}\label{com4}
			[C(r,n)(\eta_1,\eta_2)](\xi)=\eta_1(n[\xi,\rs\eta_2]-[n\xi,\rs\eta_2])-\eta_2(n[\xi,\rs\eta_1]-[n\xi,\rs\eta_1])\,.
		\end{equation}
	\end{remark}		
	\begin{proposition}\label{nrr}
		Let $r\in \wedge^2{\mathfrak g}$ be a solution of the classical Yang-Baxter equation on the Lie algebra $({\mathfrak g},[\cdot,\cdot])$ and $({\mathfrak g},[\cdot,\cdot],\delta_r)$ the coboundary Lie bialgebra defined by $r$. If $n: \g\to  \g$ is a Nijenhuis operator on $({\mathfrak g},[\cdot,\cdot])$, then under the assumptions of the Proposition \ref{nr}, $n$ defines an almost NL bialgebra $({\mathfrak g},[\cdot,\cdot],\delta_r,n)$.
	\end{proposition}
	\begin{proof}	
		We need to prove first, $[\cdot,\cdot]_{nr}$ is a Lie bracket and second, $\delta_r$ is a $1$-cocycle in the deformed cohomology by $n$, i.e. $\partial \delta^n_r\equiv 0$. The first is fulfilled since by Cor.~\ref{conn}, $nr\in \wedge^2\g$ is an $r$-matrix on $(\mathfrak g,[\cdot,\cdot])$, i.e. 
$$n\rs[\eta_1,\eta_2]_{nr}=n\rs[\eta_1,\eta_2]_{r}^{{}^tn}=n[\rs\eta_1,\rs \eta_2]_{n}=[n\rs\eta_1,n\rs\eta_2]\,.$$
		
		\noindent Therefore, the deformation of the $1$-cocycle $\delta_r$ by ${}^tn$, $\delta_r^{{}^tn}=\delta_{nr}$, is a $1$-cocycle and so $(\mathfrak g,[\cdot,\cdot], \delta_{nr})$ is a coboundary Lie bialgebra. According to the Proposition \ref{prop-equiv}, $\delta_r$ is also a $1$-cocycle in the deformed cohomology. Thus $(\mathfrak g,[\cdot,\cdot]_n, \delta_r^n)$ is also a Lie bialgebra but not necessarily coboundary.		
			\end{proof}
			
\begin{comment}
	In the above Proposition, one can also observe that 
$(\mathfrak g,[\cdot,\cdot]_n, \delta_r^n)$ is a Lie bialgebra, as demonstrated by direct computation on
		\begin{equation}\label{deltanr}
			\delta_r^n([\xi_1,\xi_2]_n)(\eta_1,\eta_2)=[\eta_1,\eta_2]_r[\xi_1,\xi_2]_n\,.
		\end{equation}
		
		\noindent Actually by plugging (\ref{rbracket}), (\ref{bi-alg}), (\ref{deltar}) and (\ref{defn}) in (\ref{deltanr}) and using (\ref{conn2}) for
		$$
		\begin{array}{rcl}
			&&	[C(r,n)(\ad^*_{\xi_1}\eta_1,\eta_2)](\xi_2)=
			[C(r,n)(\eta_1,\ad^*_{\xi_1}\eta_2)](\xi_2)=\\[3pt]
			&&
			[C(r,n)(\ad^*_{\xi_2}\eta_1,\eta_2)](\xi_1)=	[C(r,n)(\eta_1,\ad^*_{\xi_2}\eta_2)](\xi_1)	=0\,,
		\end{array}
		$$
		\noindent one gets (\ref{deltanr}). 
		\end{comment}
	\begin{corollary}
		Under the assumptions of Proposition \ref{nr}, the relation between $1$-cocycles $\delta_r$ and $\delta_{nr}$ is given by the following expression
		\begin{equation}\label{nm-r}
		\delta_{nr}(\xi)
		(\eta_1,\eta_2)=\delta_{r}(\xi)({}^tn\eta_1,\eta_2)+\delta_{r}(\xi)(\eta_1,{}^tn\eta_2)-\delta_r(n\xi) (\eta_1,\eta_2)\,.
	\end{equation} 
	\end{corollary}
	\begin{proof}
		Indeed, 
		\[
		\begin{array}{rcl}
			&&\delta_{r}(\xi)({}^tn\eta_1,\eta_2)+\delta_{r}(\xi)(\eta_1,{}^tn\eta_2)-\delta_{r}(n\xi)(\eta_1,\eta_2)=\\[5pt]
			&&	\lcf \xi,r\rcf ({}^tn\eta_1,\eta_2)
			+
			\lcf \xi,r\rcf(\eta_1,{}^tn\eta_2)-\lcf n\xi,r\rcf(\eta_1,\eta_2)=\\[5pt]
			&&	 {}^tn\eta_1([\xi,\rs\eta_2])-\eta_2([\xi,\rs {}^tn\eta_1])
			+\eta_1([\xi,\rs {}^tn\eta_2])-{}^tn\eta_2([\xi,\rs\eta_1]-\eta_1([n\xi,\rs\eta_2])+\eta_2([n\xi,\rs\eta_1]\,.
		\end{array}\]	
		Using  $n\circ \rs=\rs\circ {}^tn\,$, we obtain  
		\[
		\begin{array}{rcl}
			&&\delta_{r}(\xi)({}^tn\eta_1,\eta_2)+\delta_{r}(\xi)(\eta_1,{}^tn\eta_2)-\delta_{r}(n\xi)(\eta_1,\eta_2)=\\[5pt]
			&&\eta_1[\xi,n\rs\eta_2]-\eta_2([\xi,n\rs\eta_1]+\eta_1(n[\xi,\rs\eta_2]-[n\xi,\rs\eta_2])
			+\eta_2([n\xi,\rs\eta_1] -n[\xi,\rs\eta_1])\,.
		\end{array}\]
		Using (\ref{com4}), we  deduce that $$\eta_1(n[\xi,\rs\eta_2]-[n\xi,\rs\eta_2])+\eta_2([n\xi,\rs\eta_1] -n[\xi,\rs\eta_1])=0\,,$$ 
		and hence (\ref{nm-r}).	
	\end{proof}
	
	Now, suppose that $r, r'\in\wedge^2 {\g}$ are two $r$-matrices, and that $r$ is non-degenerate, meaning that the map $\rs:{\g}^*\to {\mathfrak g}$ is an isomorphism. We can then consider the linear map $n: {\mathfrak g}\to {\mathfrak g}$ defined by
	$$n=\rs'\circ \rs^{-1}\quad \Longrightarrow \quad \rs'=n\circ \rs=\rs\circ {}^tn\,.$$
	
	\begin{proposition}\label{compatible}
		{  Suppose that $r$ and $r'$ satisfy}  classical Yang-Baxter equations, i.e.  $\lcf r,r\rcf=0,$ $\lcf r',r'\rcf=0$, and that  $r$ is non-degenerate. If $r$ and $r'$ are compatible, i.e. $\lcf r',r\rcf=0$, then the linear map
		$n=\rs'\circ \rs^{-1}$ is  a Nijenhuis operator on $\mathfrak g\,$. 
	\end{proposition}
	\begin{proof}
		We will show the following relation
		\begin{equation}\label{rn}
			\eta(\lcf n,n\rcf(\xi_1,\xi_2))=-2 \lcf r',r\rcf (^tn\eta, \rs^{-1}\xi_1, \rs^{-1}\xi_2),
		\end{equation}
		for all $\eta\in \mathfrak g^*$ and $\xi_i\in \mathfrak g\,.$ It will prove the Proposition. Indeed, since $\lcf r,r\rcf=0$, and by applying equation (\ref{corchete}) and Proposition \ref{r}, we can deduce that
		$$
		\begin{array}{rcl}
			\eta([n\xi_1,n\xi_2])&=&{ [\rs^{-1}\xi_1, \rs^{-1}\xi_2]^{r'}(\rs'\eta)}\\[5pt]
			&=&\rs^{-1}\xi_1([\rs'\eta,\rs'(\rs^{-1}\xi_2)]) -{ \rs^{-1}\xi_2([ \rs\eta,\rs(\rs^{-1}\xi_1)])}\\[5pt]
			&=& \rs^{-1}\xi_1([\rs'\eta, n\xi_2]) - \rs^{-1}\xi_2([\rs'\eta, n\xi_1])\,.
		\end{array}
		$$
		On the other hand, using that $\lcf r,r\rcf=\lcf r',r'\rcf=0$, again  (\ref{corchete}) and Proposition \ref{r}
		we have
		$$
		\begin{array}{rcl}
			\eta(n^2[\xi_1,\xi_2])&=& n[\xi_1,\xi_2](^tn\eta)={ [\rs^{-1}\xi_1, \rs^{-1}\xi_2]^{r}(\rs'(^tn\eta))}\\[5pt]
			&=& \rs^{-1}\xi_1([ \rs'(^tn\eta), \xi_2]) - \rs^{-1}\xi_2([\rs'(^tn\eta), \xi_1])\,.
		\end{array}
		$$
		Finally, 
		$$\begin{array}{rcl}
			2 \lcf r',r\rcf ((^tn\eta), \rs^{-1}\xi_1, \rs^{-1}\xi_2)&=&^tn\eta([n\xi_1,\xi_2]) + {}^tn\eta([\xi_1,n\xi_2]) + \rs^{-1}\xi_2([\rs'^tn\eta, \xi_1])
			\\[5pt]&+&{ \rs^{-1}\xi_2([\rs'\eta, n\xi_1])} + \rs^{-1}\xi_1([n\xi_2,\rs'\eta] + \rs^{-1}\xi_1([\xi_2,\rs'{}^tn\eta])\,.
		\end{array}
		$$
		From (\ref{bracket--n}), we conclude (\ref{rn}).
	\end{proof}
	
	\noindent Under the conditions of Proposition \ref{compatible},  the $1$-cocycles $\delta_r$ and $\delta_{r'}$ associated with the coboundary Lie bialgebroids $(\mathfrak g, [\cdot,\cdot],\delta_r)$ and 
	$(\mathfrak g, [\cdot,\cdot],\delta_{r'})$, respectively, are related as follows:
    
	\noindent First, by definition
	\[
	\delta_{r'}(\xi)(\eta_1,\eta_2)=\lcf \xi,r'\rcf(\eta_1,\eta_2)= \eta_1([\xi,\rs'\eta_2])-\eta_2([\xi,\rs'\eta_1])\,,
	\]
	for all $\xi \in {\mathfrak g}$ and $\eta_1,\eta_2\in {\mathfrak g}^*$. On the other hand, 
	\[
	\begin{array}{rcl}
		&&\delta_{r}(\xi)({}^tn\eta_1,\eta_2)+\delta_{r}(\xi)(\eta_1,{}^tn\eta_2)-\delta_{r}(n\xi)(\eta_1,\eta_2)=\\[5pt]
		&&\lcf \xi,r\rcf ({}^tn\eta_1,\eta_2)+ \lcf \xi,r\rcf(\eta_1,{}^tn\eta_2)-\lcf n\xi,r\rcf(\eta_1,\eta_2)=\\[5pt]
		&& {}^tn\eta_1([\xi,\rs\eta_2])-\eta_2([\xi,\rs {}^tn\eta_1])+\eta_1([\xi,\rs {}^tn\eta_2])-{}^tn\eta_2([\xi,\rs\eta_1]))-\eta_1([n\xi,\rs\eta_2])+\eta_2([n\xi,\rs\eta_1]\,.
	\end{array}\]	
	Using $\rs'=n\circ \rs=\rs\circ {}^tn$, it reads 
	\[
	\begin{array}{rcl}
		&&\delta_{r}(\xi)(^tn\eta_1,\eta_2)+\delta_{r}(\xi)(\eta_1,^tn\eta_2)-\delta_{r}(n\xi)(\eta_1,\eta_2)=\\[5pt]
		&&\eta_1[\xi,\rs'\eta_2]-\eta_2([\xi,\rs'\eta_1]+\eta_1(n[\xi,\rs\eta_2]-[n\xi,\rs\eta_2])
		+\eta_2([n\xi,\rs\eta_1] -n[\xi,\rs\eta_1])\,.
	\end{array}\]
	Using the following equalities 
	$$\lcf r',r\rcf(\eta_1,\eta_2,\rs^{-1}\xi)=\lcf r,r\rcf(^tn\eta_1,\eta_2,\rs^{-1}\xi)=\lcf r,r\rcf(\eta_1,^tn\eta_2,\rs^{-1}\xi)=0\,,$$
	we deduce that 
	$$\eta_1(n[\xi,\rs\eta_2]-[n\xi,\rs\eta_2])+\eta_2([n\xi,\rs\eta_1] -n[\xi,\rs\eta_1])=0\,.$$
	
	\noindent Note that the latter equality is just the concomitant (\ref{com4}). Therefore, 
	\[
	\delta_{r'}(\xi)
	(\eta_1,\eta_2)=\delta_{r}(\xi)(^tn\eta_1,\eta_2)+\delta_{r}(\xi)(\eta_1,^tn\eta_2)-\delta_r(n\xi) (\eta_1,\eta_2).
	\] 
	\begin{corollary}
Let $r$ and $r'$ be elements of $\wedge^2 {\mathfrak g}$, both serving as $r$-matrices on the Lie algebra $({\mathfrak g},[\cdot,\cdot])$, where $r$ is non-degenerate. The coboundary Lie bialgebra $({\mathfrak g},[\cdot,\cdot],\delta_r)$ defined by $r$, together with the linear map $n:=r'\circ r^{-1}$, defines an almost NL bialgebra $({\mathfrak g},[\cdot,\cdot],\delta_r,n)$.
	\end{corollary}
	
	\noindent Note that in the next section, we will demonstrate that coboundary almost NL bialgebras are, in fact, coboundary NL bialgebras.
	
	\section{The hierarchy of structures on an NL bialgebra}\label{sec5}

\noindent	Given an NL bialgebra $({\mathfrak g},[\cdot,\cdot],\delta,n)$, in this section we show that, in a certain condition, one can derive a hierarchy of Lie bialgebras which are pairwise compatible in the sense of
	Proposition\ref{Lie}: any linear combination of the corresponding Lie brackets defines a new
	Lie bracket. To achieve this, we will discuss the following hierarchies utilizing an (almost) Nijenhuis operator:
	\begin{itemize}
		\item The hierarchy of deformed (Lie) brackets on a Lie algebra,
		\item The hierarchy of adjoint representations of deformed Lie brackets,  
		\item The hierarchy of deformed $1$-cochains of the $1$-cocycle $\delta$, encompassing all three types of deformation.
	\end{itemize}
\subsection{The hierarchy of deformed brackets and adjoint representations on a Lie algebra} 	
Given a Nijenhuis operator $n:\g\to \g$ on a Lie algebra $(\g,[\cdot,\cdot])$ (and correspondingly, $^tn:\g^*\to\g^*$ on its dual $(\g^*,[\cdot,\cdot]_{\g^*})$), there exist a hierarchy of deformed Lie brackets on both $\g$ and $\g^*$ which are compatible in pairs \cite{KoMa}. This leads to a hierarchy of adjoint representations on $\g$ for the iterated Lie brackets.  

 Note that, throughout this section, for the sake of simplicity, we denote the deformed brackets on $(\g,[\cdot,\cdot])$ iterated by any power $n^i$, for $(i>1)$ as $[\cdot,\cdot]_{i}$ and the corresponding deformed brackets on $\g^*$ iterated by $({}^tn)^i$ as $[\cdot,\cdot]^{i}$.

\begin{definition} (\cite{KoMa})
Suppose $n$ be a Nijenhuis operator on $(\g,[\cdot,\cdot])$. The Nijenhuis torsion of $n$ with respect to the deformed Lie bracket $[\cdot,\cdot]_{i}$ is given by
\[
\lcf n,n\rcf_{i} (\xi_1,\xi_2)=\lcf n,n\rcf_{{i-1}} (n\xi_1,\xi_2)+\lcf n,n\rcf_{{i-1}} (\xi_1,n\xi_2)-n\lcf n,n\rcf_{{i-1}} (\xi_1,\xi_2)\,.
\]
\end{definition}
	
	\begin{remark} (\cite{KoMa})
		If  $\lcf n,n\rcf=0$, then for every integer $i$:  
	\begin{itemize}
		\item 	The vanishing of Nijenhuis torsion of $n$ is preserved in a hierarchy, i.e. $\lcf n,n\rcf_{i}=0$.
		\item The Nijenhuis torsion of $n^k$ with respect to the Lie algebra structure $(\g,[\cdot,\cdot]_i)$ vanishes, i.e. $\lcf n^{k},n^{k}\rcf_{i}=0\,.$
	\end{itemize}	
	
	\end{remark} 

  	More generally, we have:
	\begin{remark}\label{remn} (\cite{KoMa}, Lemma 5.1)
		Let $\lcf n,n\rcf=0$ and $\lcf {}^tn,{}^tn\rcf=0$. For each integer $k \geqq 0$, and for each integer $i \geqq 0$, $\xi_1,\xi_2\in \mathfrak g$ and $\eta_1,\eta_2\in \mathfrak g^*$, 
		\begin{equation}\label{hinnt}
			\begin{array}{rcl}
				n^{i}[\xi_1, \xi_2]_{{k+i}}&=&\left[n^{i} \xi_1, n^{i} \xi_2\right]_{k}, \\[5pt]
				(^tn)^{i}[\eta_1, \eta_2]^{{k+i}}&=&\left[({}^tn)^{i} \eta_1, ({}^tn)^{i} \eta_2\right]^{k}\,.
			\end{array}
		\end{equation}
		\noindent 
	\end{remark}
\noindent When $n$ is a Nijenhuis operator, we obtain the following recurrence relation among the deformed brackets:
	\begin{equation}\label{nk}
		\begin{array}{rcl}	[\xi_1,\xi_2]_{n^i}&=&[n^i\xi_1,\xi_2]+[\xi_1,n^i\xi_2]-n^i[\xi_1,\xi_2]\\&=&[n^{i-1}\xi_1,\xi_2]_n+[\xi_1,n^{i-1}\xi_2]_n-n^{i-1}[\xi_1,\xi_2]_n\,, 
		\end{array} 
	\end{equation}
	\noindent and 	\begin{equation}\label{nk*}
		\begin{array}{rcl}
			[\eta_1,\eta_2]^{(^tn)^j} &=& [(^tn)^j\eta_1,\eta_2]_{\g^*} +[\eta_1, (^tn)^j\eta_2]_{\g^*}- (^tn)^j[\eta_1,\eta_2]_{\g^*}\\[3pt]
			&=& [(^tn)^{j-1}\eta_1,\eta_2]^{^tn} +[\eta_1, (^tn)^{j-1}\eta_2]^{^tn}-(^tn)^{j-1}[\eta_1,\eta_2]^{^tn}\,.
		\end{array}
	\end{equation}
	\noindent For instance, in the case $i=1$ and $k=1$, we have
	\begin{equation}\label{adn2}
		[\xi_1,\xi_2]_{n^2}=[n^2\xi_1,\xi_2]+[\xi_1,n^2\xi_2]-n^2[\xi_1,\xi_2]=[n\xi_1,\xi_2]_n+[\xi_1,n\xi_2]_n-n[\xi_1,\xi_2]_n\,.
	\end{equation}
	\begin{remark}\label{weaknn}
		If $n$ is an almost Nijenhuis structure on a Lie algebra $(\g,[\cdot,\cdot])$, there is no obvious recurrence relation between deformed brackets, i.e.
		\[
		\begin{array}{rcl}	[\xi_1,\xi_2]_{n^i}&=&[n^i\xi_1,\xi_2]+[\xi_1,n^i\xi_2]-n^i[\xi_1,\xi_2]\\&\neq&[n^{i-1}\xi_1,\xi_2]_n+[\xi_1,n^{i-1}\xi_2]_n-n^{i-1}[\xi_1,\xi_2]_n\,, 
		\end{array} 
		\]
		\[
		\begin{array}{rcl}
			[\eta_1,\eta_2]^{(^tn)^j} &=& [(^tn)^j\eta_1,\eta_2]^* +[\eta_1, (^tn)^j\eta_2]^*- (^tn)^j[\eta_1,\eta_2]^*\\[3pt]
			&\neq& [(^tn)^{j-1}\eta_1,\eta_2]^{^tn} +[\eta_1, (^tn)^{j-1}\eta_2]^{^tn}-(^tn)^{j-1}[\eta_1,\eta_2]^{^tn}\,.
		\end{array}
		\]
	\end{remark}
	\noindent Note that we can consider a linear map $n^k:\g\to \g$ as a $\g$-valued $1$-form on the Lie algebra $\g$. The coboundary of this operator with respect to the initial Lie bracket is then given by:
	\[
	\partial n^k(\xi_1,\xi_2)=[\xi_1,\xi_2]_{n^k}\,.
	\]
	\begin{corollary}\label{weakn}
		\noindent The Nijenhuis torsion 
		$$\lcf n^k,n^k\rcf (\xi_1,\xi_2)=n^k[\xi_1,\xi_2]_{n^k}-[n^k\xi_1,n^k\xi_2]\,,$$
		can be expressed in terms of the Nijenhuis torsion of previous orders,  with respect to the initial Lie bracket $[\cdot,\cdot]$, as follows 
		\[
		\begin{array}{rcl}
			\lcf n^k,n^k\rcf (\xi_1,\xi_2)&=&n^{k-1}\lcf n,n\rcf[ (n^{k-1}\xi_1,\xi_2)+(\xi_1,n^{k-1}\xi_2)]+\lcf n^{k-1},n^{k-1}\rcf (n\xi_1,n\xi_2)\\[5pt]
			&+&n^2\lcf n^{k-1},n^{k-1}\rcf(\xi_1,\xi_2)-n^2\lcf n^{k-2},n^{k-2}\rcf(n\xi_1,n\xi_2)\,.
		\end{array}
		\]
        	\end{corollary}
\noindent		For example, when $k=2$, the relation reads:
		\[
		\lcf n^2,n^2\rcf (\xi_1,\xi_2)=n\lcf n,n\rcf[ (n\xi_1,\xi_2)+(\xi_1,n\xi_2)]+\lcf n,n\rcf (n\xi_1,n\xi_2)+n^2\lcf n,n\rcf (\xi_1,\xi_2)\,.\
		\]
	\begin{definition}
		  The Nijenhuis concomitant $[n^i,n^j]$ (\cite{Ni},\cite{RaReHa}) of two Nijenhuis operators $n$ and $n' $ on a Lie algebra $(\mathfrak g,[\cdot,\cdot])$ is defined by
			\[	\begin{array}{rcl}	\lcf n,n' \rcf(\xi_1,\xi_2)& =& 	[ n\xi_1,n'\xi_2] +[n'\xi_1,n\xi_2] - n[n'\xi_1, \xi_2] -n'[n\xi_1, \xi_2]\\[4pt]
				&-&   n[\xi_1,n'\xi_2] - n'[\xi_1,n\xi_2] + n\circ n'[\xi_1,\xi_2] + n'\circ n[\xi_1,\xi_2]\,,	\end{array}\]
		  for every pair of basis elements 
		  $\xi_1,\xi_2$ of $\mathfrak g\,$.
		\begin{comment}
		Two Nijenhuis structures are considered compatible if their Nijenhuis concomitant vanishes. When $\lcf n,n \rcf\equiv 0$, there exists a compatible hierarchy of Nijenhuis structures $n^k$ for all $k>1$. Consequently, the relation $\lcf n^i,n^j \rcf \equiv 0$ holds for any integers $i,j$.
		\end{comment}

	\end{definition}
	Using Remark \ref{remn}, we deduce the following.
	\begin{corollary}
Suppose $n$ is a Nijenhuis operator on a Lie algebra $(\g,[\cdot,\cdot])$. We have the following hierarchy of Lie brackets and hierarchy of adjoint representations $\ad^{n^i}_{\xi}:\g\to\g$, $i \geqq 0$ and $\xi\in\g$: 
		\[\begin{array}{rclrclrcl}
			&(\g,[\cdot,\cdot])&  \, &\rightsquigarrow&\, &\ad_{\xi}&\\[4pt]
			
			&(\g,[\cdot,\cdot]_n)&  \, &\rightsquigarrow&\, &\ad^n_{\xi}=[\ad_{\xi},n]+\ad_{n\xi}&\\[4pt]
			
			&(\g,[\cdot,\cdot]_{n^2})&  \, &\rightsquigarrow&\, &\ad^{n^2}_{\xi}=[\adn_{\xi},n]+\adn_{n\xi}&\\[4pt]
			&\quad\vdots& && &\vdots& \\[4pt]
			&(\g,[\cdot,\cdot]_{n^i}) & \, &\rightsquigarrow&\, &\ad^{n^i}_{\xi}=[\ad^{n^{i-1}}_{\xi},n]+\ad^{n^{i-1}}_{n\xi}&	\end{array}	\]
	\end{corollary}
	\noindent For example for $i=2$
	\[\begin{array}{rcl}
		\ad^{n^2}_{\xi}&=&[[\ad_{\xi},n],n]+2[\ad_{n\xi},n]+\ad_{n^2\xi}\,,\\[3pt]
		(\ad^{n^2}_{\xi})^*&=&[^tn,[^tn,\ad^*_{\xi}]]+2[^tn,\ad^*_{n\xi}]+\ad^*_{n^2\xi}\,.
	\end{array}\]
Using $\ads_{\,n\xi}=n\circ\ads_{\,\xi}$ (see Lemma \ref{lemma.rel1}), it can be rewritten as
	\[\begin{array}{rcl}\ad^{n^2}_{\xi}&=&[ad_{\xi},n]n+[\ad_{n\xi},n]+\ad_{n^2\xi}\,, \\[3pt]
		(\ad^{n^2}_{\xi})^*&=&^tn[^tn,ad^*_{\xi}]+[^tn,\ad^*_{n\xi}]+\ad^*_{n^2\xi}\,.
	\end{array}\]
	\begin{remark}
    For a Nijenhuis operator $n$ on a Lie algebra $(\mathfrak{g},[\cdot,\cdot])$, from equation (\ref{nk}), $\ad^{n^i}_{\xi}$ can be expressed in terms of $\adn_{\xi}$ for any integer $i$ as follows:
		\begin{equation}\label{hierad}
		\ad^{n^i}_{\xi}=[\adn_{\xi},n^{i-1}]+\adn_{n^{i-1}\xi}\,.
		\end{equation}
	\end{remark}
	\begin{remark}
		If  $\lcf n,n\rcf=0$ and $\lcf {}^tn,{}^tn\rcf=0$, then for any $i,j$
		$$C(n^i,\ad^{n^i})\equiv 0 \quad \&\quad C((^tn)^j,(\ad^{n^j})^*)\equiv 0\,.$$
	\end{remark}
	
\subsection{The hierarchy of deformed one-cochains on a Lie algebra}	Now, we consider the hierarchy  of deformed $1$-cochains in the cohomology defined by $\ad^{(2)}:\g \times \wedge^2 \g \to \wedge^2 \g$ on the Lie algebra $(\g,[\cdot,\cdot])$. The deformation of a $1$-cocycle $\delta$ by means of $(^tn)^k$ is represented by a linear map $\delta_{n^k}:{\mathfrak g}\to \wedge^2{\mathfrak g}$, given by: 
	\begin{equation}\label{Dni}
		\delta_{(n^{t})^k}(\xi)=\iota_{{(^tn)}^k}\delta(\xi)-\delta(n^k\xi)\,, \mbox{for}\quad k\in \mathbb N\,.\end{equation}

	\begin{lemma}
		The hierarchy of iterated deformed $1$-cochains (\ref{Dni}) by means of $(^tn)^k$, is given by 
		\begin{equation}\label{Dnii}
			\delta_{(^tn)^k}(\xi)=\iota_{^tn}\delta_{(^tn)^{k-1}}(\xi)-\delta_{(^tn)^{k-1}}(n\xi)\,.\end{equation} 
	\end{lemma}
	\begin{proof}
Using \ref{phiii}, we get
	\begin{equation}
		\begin{array}{rcl}
			&&\delta_{(n^{t})^k}(\xi)(\eta_1,\eta_2)=\iota_{{(^tn)}^k}\delta(\xi)(\eta_1,\eta_2)-\delta(n^k\xi)(\eta_1,\eta_2)\\[6pt]
			&&=	\iota_{{}^tn}(\iota_{{}^tn^{k-1}}\delta(\xi))(\eta_1,\eta_2)-\delta(\xi)\left[  ({}^tn\eta_1,{}^tn^{k-1}\eta_2)+({}^tn^{k-1}\eta_1,{}^tn\eta_2)\right] -\delta(n^k\xi)(\eta_1,\eta_2)\,.\\[6pt]
			
		\end{array}	
	\end{equation}
	Now, by plugging $\iota_{{}^tn^{k-1}}\delta(\xi)(\eta_1,\eta_2)=\delta_{{}^tn^{k-1}}(\xi)(\eta_1,\eta_2)+\delta(n^{k-1}\xi)(\eta_1,\eta_2)$ in the above equality, and then using $\lcf {}^tn^{k-1}, {}^tn^{k-1}\rcf \equiv 0$ and $\lcf {}^tn, {}^tn\rcf_{{}^tn^{k-1}} \equiv 0$, we arrive at (\ref{Dnii}).
	\end{proof}

\noindent In the following, we present results concerning the interplay between, the adjoint representation of the Lie algebra $(\g,[\cdot,\cdot])$, the 1-cocycle $\delta$, and the Nijenhuis structures $n^k$ on $\g$, for all $k>1\,$. 
	\begin{proposition}\label{general.com}
		For a Nijenhuis operator $n$ on the Lie algebra $(\g,[\cdot,\cdot])$, we have:  
		\begin{equation}\label{rel11n}
	{\ads^*_{\,n^{k}\xi}}	= \ads_{\,\xi}^*\circ  (^tn)^k\,,\quad \forall\xi \in \g \,.
		\end{equation}
	\end{proposition}	
	\begin{proof}
From the equation
		 $\lcf n,n\rcf (n^{k-1}\xi,\xi_1)=0\,,$ we can deduce that $\ads_{n^k\xi}=n\ads _{n^{k-1}\xi}$ on $\g$. Consequently, we have
$
\ads^*_{n^k\xi}=	\ads^* _{n^{k-1}\xi}\circ {}^tn$ on $\g^*$. By applying this process iteratively, we ultimately arrive at (\ref{rel11n}).
	\end{proof}	 	 
		  \begin{proposition}\label{prop2}
		  	Let $(\g,[\cdot,\cdot],\delta,n)$ be an NL bialgebra. Then
		  \begin{equation}\label{rel111}
		  \left[ 	\iota_{\ads^*_{\,\xi_1}}\delta(n^k\xi_2)-\iota _{{}^tn\circ \ads^*_{\,\xi_1}}\delta(n^{k-1}\xi_2)
		  \right](\eta_1,\eta_2)~-~\circlearrowleft~=\,0\,,
		  \end{equation}
		  where by $\circlearrowleft$ we mean the permutation of $\xi_1$ and $\xi_2\,$.
		  	  \end{proposition}
		  	  \begin{proof}	  	
		\noindent  Since $\lcf{}^tn,{}^tn\rcf=0$, we have 
		$$[\iota_{\ads^*_{\xi_1}}\lcf {}^tn,{}^tn \rcf (\eta_1,\eta_2)](n^{k-1}\xi_2)~-~\circlearrowleft~=0\,,$$
	and thus
	\[
	\begin{array}{rcl}
&& \left[	\iota_{\ads^*_{\,\xi_1}}\delta(n^k\xi_2)-\iota _{{}^tn\circ \ads^*_{\,\xi_1}}\delta(n^{k-1}\xi_2)
	\right](\eta_1,\eta_2)~-~\circlearrowleft~=\\[8pt]
	&&\delta(n^{k-1}\xi_2)\left[({}^tn\eta_1,\ads^*_{\xi_1}\eta_2)+(\ads^*_{\xi_1}\eta_1,{}^tn\eta_2)	\right]-\iota_{\ads^*_{\xi_1}}\dn(n^{k-1}\xi_2)(\eta_1,\eta_2)~-~\circlearrowleft\,.
	\end{array}
	\]
	Since $C(\delta,n)=0$, hence (\ref{rel111}). 
		\end{proof} 
 
	\begin{proposition}\label{Propositionequiv.n}
		Let $(\g,[\cdot,\cdot],\delta,n)$ be an NL bialgebra. The deformation of the $1$-cocycle $\delta$ by $(^tn)^k$ is a $1$-cocycle if and only if $\delta$ is also a $1$-cocycle in the deformed cohomology defined by the deformed Lie bracket $[\cdot,\cdot]_{{n}^k}$\,. 
	\end{proposition}
	\begin{proof}
 By definition and using (\ref{cocycle}), (\ref{Dnii}), (\ref{adnnad}), (\ref{rel2}), we have
		\[
\partial	\delta_{^tn^k}(\xi_1,\xi_2)=\iota_{[{}^tn^k,\ad^*_{\xi_1}]}	\delta(\xi_2)-\iota_{[{}^t{n^k},\ad^*_{\xi_2}]}	\delta(\xi_1)-\iota_{\ad^*_{\xi_1}} \delta(n^k\xi_2)+\iota_{\ad^{*}_{\xi_2}}\delta(n^k\xi_1)+\delta(n^k[\xi_1,\xi_2])\,.
		\]
		On the other hand
\[ 
\partial	\delta^{n^k}(\xi_1,\xi_2)=\iota_{[{}^tn^k,\ad^*_{\xi_1}]}	\delta(\xi_2)-\iota_{[{}^t{n^k},\ad^*_{\xi_2}]}	\delta(\xi_1)+\iota_{\ad^*_{n^k\xi_1}} \delta(\xi_2)-\iota_{\ad^*_{n^k\xi_2}} \delta(\xi_1)+\delta([\xi_1,\xi_2]_{n^k})\,.
\]	
Now, using the fact that $\partial\delta(n^k\xi_1,\xi_2)=\partial\delta(\xi_1,n^k\xi_2)=0$ we obtain: $$\partial	\delta_{^tn^k}\equiv 0 \iff \partial	\delta^{n^k}\equiv 0\,. \mbox{\qedhere}$$	\end{proof}
	\subsection{The hierarchy of Lie bialgebras on an NL bialgebra}\label{sec5.2}
Given an NL bialgebra, we will demonstrate that, under certain conditions, it is possible to generate a compatible hierarchy of Lie bialgebras.
By a compatible hierarchy of Lie bialgebras, we refer to Lie bialgebras  $(\mathfrak g,[\cdot,\cdot]_{n^i}, \delta^{n^i}_{(^tn)^j})\,$, which represent the three types of deformations of $1$-cocycles. The maps $\delta^{n^i}_{(^tn)^j}$ are $1$-coycles in the cohomology induced by $(\ad^{n^i})^{(2)}:  \g\times \wedge^2\g\to \wedge^2\g$ of the Lie algebra $(\g, [\cdot,\cdot]_{n^i})\,.$
 We denote  $\delta^{n^i}_{(^tn)^j}$ as the deformed $1$-cocycle by means of $(n^i,(^tn)^j)$.  
\begin{figure}[H]
    \centering
    \subfloat{
        \xymatrixcolsep{3pc}
        \xymatrix@R-3mm @C-7mm{
            & ((\g,[\cdot,\cdot]),(\g^*,[\cdot,\cdot]_{\g^*})) \ar@{.>}[dl]_{n} \ar@{.>}[dr]^{^tn} & \cr
            ((\g,[\cdot,\cdot]_n),(\g^*,[\cdot,\cdot]_{\g^*})) \ar@{.>}[d]_{n} \ar@{.>}[dr]^{^tn} & & ((\g,[\cdot,\cdot]),(\g^*,[\cdot,\cdot]^{^tn})) \ar@{.>}[d]^{^tn} \ar@{.>}[dl]_{~~~~n} \cr
            & ((\g,[\cdot,\cdot]_n),(\g^*,[\cdot,\cdot]^{^tn})) \ar@{.>}[dl]_{} \ar@{.>}[dr]^{} & \cr
            ((\g,[\cdot,\cdot]_{n^2}),(\g^*,[\cdot,\cdot]_{\g^*})) \ar@{.>}[ur]_{} & & ((\g,[\cdot,\cdot]),(\g^*,[\cdot,\cdot]^{{(^tn)}^2})) \ar@{.>}[ul]_{} \cr
            \vdots & \vdots & \vdots \cr
            & & \cr
            \ar@{.>}[d]_{n} \ar@{.>}[dr]^{^tn} & & \ar@{.>}[dl]_{~~~~n} \ar@{.>}[d]^{^tn} \cr
            & ((\g,[\cdot,\cdot]_{n^{i}}),(\g^*,[\cdot,\cdot]^{{(^tn)}^{k-i}})) \ar@{.>}[dl]_{} \ar@{.>}[dr]^{} & \cr
            ((\g,[\cdot,\cdot]_{n^{k}}),(\g^*,[\cdot,\cdot]_{\g^*})) \ar@{.>}[ur]_{} & & ((\g,[\cdot,\cdot]),(\g^*,[\cdot,\cdot]^{{(^tn)}^{k}})) \ar@{.>}[ul]_{} 
        }
         }
    \caption{The diagram illustrates all the deformations under consideration. The left-hand column displays the pairs of Lie algebras corresponding to the iterated 1-cocycles $\delta^{n^k}$ in the deformed cohomology defined on the Lie algebras $(\g,[\cdot,\cdot]_{n^k})$. The right-hand column shows the pairs of Lie algebras corresponding to 1-cocycles deformed by means of $(^tn)^k$. The central column represents the double deformations.}
    \label{fig:deformations}
\end{figure}	
In the subsequent discussion, we will examine the conditions on NL bialgebras under which they produce a hierarchy of Lie bialgebras. 
 		\begin{corollary} \label{equivad}
	Let $(\g,[\cdot,\cdot],\delta,n)$ be an almost NL bialgebra for which 
	 the operator $n^t$ is $\ad^*$-equivariant on the Lie algebra $(\g^*,[\cdot,\cdot]_{\g^*})$, where $\ad^*:\g^*\times \g^*\to \g^*$ is the corresponding adjoint representation. That is,
	\begin{equation}\label{ad-equ}
		\delta(\xi)	(\eta_1,n\eta_2)=\delta(n\xi)(\eta_1,\eta_2),\quad \forall \xi \in \g,\quad \eta_1,\eta_2\in \g^*\,.
	\end{equation}
	Then the following are equivalent:
	\begin{itemize}
			\item The operator ${}^tn$ is $\ad^*$-equivariant in the Lie algebra $(\g^*,[\cdot,\cdot]^{{}^tn})$. That is
		\[
		\iota_{{}^tn^{k-1}\ads^*_{\xi_1}}\dn(\xi_2)(\eta_1,\eta_2)~-~\circlearrowleft~=~\iota_{{}^tn^{k-2}\ads^*_{\xi}}\dn(n\xi_1)(\eta_1,\eta_2)~-~\circlearrowleft\,.
		\]
	\item For each $k$, ${}^tn^k$ is $\ad^{(2)}$-equivariant. That is,
	\[
	\iota_{n^{k}\ads^*_{\xi_1}}\delta(\xi_2)(\eta_1,\eta_2)	 ~-~\circlearrowleft~=	~\iota_{ \ads_{\,\xi_1}}\delta(n^k\xi_2)(\eta_1,\eta_2) ~-~\circlearrowleft\,.
	\]
	\end{itemize}
	\end{corollary}
\noindent Note that an almost NL bialgebra $(\g,[\cdot,\cdot],\delta,n)$ for which ${}^tn$ is $\ad^*$-equivariant , is an NL bialgebra. 

The following corollaries are immediate results for NL bialgebras that lead to the theorem establishing a hierarchy of Lie bialgebras.
		\begin{corollary}\label{hire}
 	Let $(\g,[\cdot,\cdot],\delta,n)$ be an NL bialgebra. For each integer $k$, the concomitant $C(\delta,n^k)$ 	
			\begin{equation}\label{rel112}
		C(\delta,n^k)(\eta_1,\eta_2)([\xi_1,\xi_2])=	\iota_{n^{k}\ads^*_{\xi_1}}\delta(\xi_2)(\eta_1,\eta_2)		-\iota_{ \ads_{\,\xi_1}}\delta(n^k\xi_2)(\eta_1,\eta_2) ~-~\circlearrowleft	\,,\\[4pt]
	\end{equation}
	vanishes if and only if $n^t$ is an $\ad^*$-equivariant operator on the Lie algebra $(\g^*,[\cdot,\cdot]_{\g^*})$.	
 	\end{corollary}
\noindent Combining the results of Corollaries ~\ref{equivad} and \ref{hire}, we deduce that $C(\delta,n^k)\equiv 0$ if and only if $C(\dn,n^{k-1})\equiv 0$. In general:
    \begin{corollary}\label{doubledef}
	For an NL bialgebra $(\g,[\cdot,\cdot],\delta,n)$, we have
	\[
	\partial \delta_{(^tn)^k}\equiv 0\,\iff \,\partial \delta^{n^k}\equiv 0 \,\iff \,\partial \delta_{(^tn)^j}^{n^i}\equiv 0,\quad i+j=k\,.
	\]
\end{corollary}
 
	\begin{theorem}\label{main}
Let $(\g,[\cdot,\cdot],\delta,n)$ be an (almost) NL bialgebra, and let ${}^tn$ be an $\ad^*$-equivariant operator on the Lie algebra $(\g^*,[\cdot,\cdot]_{\g^*})$. Then there exists a compatible hierarchy of Lie bialgebras  given by  $(\mathfrak g,[\cdot,\cdot]_{n^i}, \delta^{n^i}_{(^tn)^j})\,$ for all integers  $i,j\geq 0$.
\end{theorem}	

	\begin{proof}
The deformed $1$-cochains $\delta_{(^tn)^j}:{\mathfrak g}\to \wedge^2{\mathfrak g}$ by means of $({}^tn)^j$, are  $1$-cocycles in the cohomology induced by the representation $\ad^{(2)}:  {\mathfrak g} \times \wedge^2{\mathfrak g} \to \wedge^2{\mathfrak g}$ of the Lie algebra $({\mathfrak g}, [\cdot,\cdot])$ if and only if   $\d^{(^tn)^{j}}=\iota_{^tn}\circ \d^{(^tn)^{j-1}}-\d^{(^tn)^{j-1}}\circ n$
		are  $1$-cocycles. 
        The relation between the deformed cochains in the hierarchy is given by
\[
\delta_{(^tn)^j}(\xi)(\eta_1,\eta_2)=\delta_{(^tn)^{j-1}}(\xi)(^tn\eta_1,\eta_2)+\delta_{(^tn)^{j-1}}(\xi)(\eta_1,^tn\eta_2)-\delta_{(^tn)^{j-1}}(n\xi)(\eta_1,\eta_2)\,, \]
for all  $\eta_i\in {\mathfrak g}$ and $\xi \in {\mathfrak g}
.$	To prove the theorem, we will show the following for $k=i+j$, $k\geq 2$: 
		\begin{enumerate}
			\item  $n^{k}$ is a Nijenhuis tensor with respect to $[\cdot,\cdot]_{n^{k-1}}$\,.
            \item $C(\delta,n^{k-1})\equiv 0\,.$	
			\item  $\iota_{^tn}\circ \d^{(^tn)^{j-1}}-\d^{(^tn)^{j-1}}\circ n:\mathfrak g\to \wedge^2\mathfrak g$ is a $1$-cocycle for the adjoint representation of $({\mathfrak g}. [\cdot,\cdot]_{n^{i}})\,.$
		\end{enumerate}
		
		\noindent (1) It is a consequences of (\ref{hinnt}).  (2) The vanishing of the concomitant is preserved by the hierarchy, as implied by the Corollary \ref{hire}.
		
	\noindent $(3)$ 	Suppose that $\partial \delta ^{n^{k-1}}=\partial \delta ^{n^{k-2}}\equiv 0$. Using (\ref{hierad}), we have
		\[
	(\ad_{\xi}^{n^i})^*=[{}^tn^{i-1},\ads^*_{\xi}]+[{}^tn^{i-1},\ad^*_{n\xi}]+\ads^*_{n^{i-1}\xi}+\ad^*_{n^i\xi}\,.
	\]
Now, we will substitute the coadjoint representation mentioned above and
	\[
\delta[\xi_1,\xi_2]_{n^k}=\delta[n\xi_1,\xi_2]_{{n^{k-1}}}+\delta[\xi_1,n\xi_2]_{{n^{k-1}}}-\delta[n\xi_1,n\xi_2]_{{n^{k-2}}}
	\] 
in 
$\partial\delta^{n^k}[\xi_1,\xi_2]=\iota_{(\ad_{\xi_1}^{n^k})^*}\delta(\xi_2)-\iota_{(\ad_{\xi_2}^{n^k})^*}\delta(\xi_1)-\delta[\xi_1,\xi_2]_{n^k}$. Using Propositions \ref{general.com}, \ref{prop2} and (2), we get $\partial \delta^{n^k}\equiv 0$.
Finally by Proposition \ref{doubledef}, we arrive at (2).
	\end{proof}

 	Another class of NL bialgebras are those for which the operator $n:\g\to\g$ is $ad$-equivariant, i.e.
	$$n\circ \ad_{\xi}=\ad_{\xi}\circ n,\quad \forall \xi \in \mathfrak g.$$
	\begin{corollary}
		Let $(\g,[\cdot,\cdot], \delta,n)$ be an almost NL bialgebra and let $n:\g\to\g$ be an $ad$-equivariant operator. If the Nijenhuis torsion of the transpose map $^tn:\g^*\to\g^*$ vanishes, then $(\g,[\cdot,\cdot], \delta,n)$ forms an NL bialgebra that generates a hierarchy of Lie bialgebras.
	\end{corollary}	 
	\begin{proof}
	If $n$ is $ad$-equivariant, then $\ads^*_{\xi}=0$ for all $\xi\in\g$. Consequently, $C(\delta,n)\equiv 0$ and $C(\delta,n^k)\equiv 0$ as well.
	\end{proof}
	\begin{example} 
		Let $\g$ be the direct sum of
		two copies of the solvable Lie algebra of dimension $2$. The nonzero Lie brackets are given by
		$$[X_1,X_2]=X_2, \quad [X_3,X_4]=X_4\,.$$
				\noindent The linear map 	
		$$ \delta(X_1)=0,\quad \delta (X_2)=2X_1\wedge X_2, \quad \delta(X_3)=X_3\wedge X_4, \quad \delta(X_4)=0\,,$$
		is a $1$-cocycle on $\g$ and $[\cdot,\cdot]_{\g^*}:=\delta^t$ defines a Lie bracket on $\g^*$ by
		\[
		[X^1,X^2]_{\g^*}=2X^2,\quad [X^3,X^4]_{\g^*}=X^3\,.
		\]		 
		The following operator  
		$$n(X_1)=X_1,\quad n(X_2)=X_2,\quad n(X_3)=0,\quad n(X_4)=0$$
		 is Nijenhuis on $\g$  and one can see that $(\mathfrak g,[\cdot,\cdot],\delta,n)$ forms an almost NL bialgebra .
        
 \noindent       Since $n$ is $ad$-equivariant, $C(\delta,n^k)\equiv 0$. Furthermore, $\lcf ^tn,^tn\rcf\equiv 0$. Consequently, NL bialgebra $(\mathfrak g,[\cdot,\cdot],\delta,n)$ generates a hierarchy of Lie bialgebras.
	\end{example}
\smallskip
 	\noindent In what follows, we show that coboundary almost NL bialgebras are examples of NL bialgebras.
	
	\begin{corollary}\label{cobocom}
		Let $r\in \wedge^2{\mathfrak g}$ be a non-degenerate solution of the classical Yang-Baxter equation of the Lie algebra ${\mathfrak g}$ and $n: \mathfrak g\to  \mathfrak g$ be a Nijenhuis operator such that
		\begin{enumerate}
			\item [(i)] $n\circ \rs=\rs\circ {}^tn\,$
			\item [(ii)] $C(r,n)(\eta_1,\eta_2)=\ads^*_{\,\rs\eta_2}\eta_1-\ads^*_{\,\rs\eta_1}\eta_2=0,\quad \forall \eta_i\in\g^*\,.$ 
		\end{enumerate}
		
		Then, $(\g,[\cdot,\cdot],\delta_r,n)$ is a coboundary NL bialgebra and it generates a hierarchy of bialgebras. 
	\end{corollary} 
	\begin{proof}
		The proof is consequence of the fact that if the pair $( r, n)$ satisfies the compatibility conditions $(i)$ and $(ii)$ then $n\rs\in \wedge ^2\g$ is again a solution of the classical Yang-Baxter equation and so $r_k=n^k\rs$ for $k\in \mathbb N$, are $r$-matrices \cite{RaReHa}. Therefore $\delta_{n^kr}$ are $1$-cocycles of the coboundary Lie bialgebras $	(\g,[\cdot,\cdot],\delta_{nr})$, and then according to the Proposition \ref{Propositionequiv.n} all pairs $	((\g,[\cdot,\cdot]_{n^i}),(\g^*,[\cdot,\cdot]^{n^jr}))$ are Lie bialgebras but not necessarily coboundary. The coboundary is just guaranteed for the Lie bialgebras of the deformed $1$-cocycles $\delta_{n^kr}$ by means of $(^tn)^k$. On the other hand, $n\rs$ is $r$-matrix means
		\[
		n\rs[\eta_1,\eta_2]^{n\rs}=[n\rs\eta_1,n\rs\eta_2]\,,
		\]
		then, from the compatibility conditions, we have
		\[
		n\rs[\eta_1,\eta_2]^{\rs}_{^tn}=n\rs([^tn\eta_1,\eta_2]^{\rs}+[\eta_1,{}^tn\eta_2]^{\rs}-{}^tn[\eta_1,\eta_2]^{\rs})=\rs[^tn\eta_1,^tn\eta_2]^r
		\]	
		and therefore
		$
	{}^tn({}[^tn\eta_1,\eta_2]^{\rs}+[\eta_1,{}^tn\eta_2]^{\rs}-{}^tn[\eta_1,\eta_2]^{\rs})=[{}^tn\eta_1,{}^tn\eta_2]^r\,,$ which implies $\lcf {}^tn,{}^tn\rcf^r =0$ and so $^tn$ is a Nijenhuis operator on the dual Lie algebra $(\g^*,[\cdot,\cdot]_r)$ of the Lie algebra $(\g,[\cdot,\cdot])$.
		
		\noindent 		Note that in this case, by (\ref{deltar}) and using a straightforward computation and the skew-symmetric property of $r$-matrix $r$, we have the following relation between the concomitant of $C(\delta_r,n)$ and the concomitant of $C(r,n)$
		\[
		\begin{array}{rcl}
			C(\delta_r,n)(\eta_1,\eta_2)([\xi_1,\xi_2])&=&C(r,n)(\ads^*_{\xi_1}\eta_1,\eta_2)(\xi_2)-C(r,n)(\eta_1,\ads^*_{\xi_1}\eta_2)(\xi_2)\\[3pt]
			&-&C(r,n)(\ads^*_{\xi_2}\eta_1,\eta_2)(\xi_1)+C(r,n)(\eta_1,\ads^*_{\xi_2}\eta_2)(\xi_1)\,.\\
		\end{array}	
		\]
	On the other hand, if we replace $\rs$ by $n^k\rs$ in (\ref{conr}), and then using the fact that $\ads^*_{n^k\xi}=\ads^*_{\xi}\circ{}^tn^k$, we get 
$
C(n^k\rs,n)(\eta_1,\eta_2)=\frac{1}{2}	\iota_{{}^tn^k}C(r,n)(\eta_1,\eta_2)\,.
	$
	It shows that the vanishing of concomitant is preserved during a hierarchy for the coboundary NL bialgebras.	
			\end{proof}
	\begin{example}
		Consider a Lie algebra on $\R^4$ whose Lie bracket is characterized by
		$$[X_1,X_4]=X_1,\quad [X_3,X_4]=X_2\,.$$
		There is a Nijenhuis structure $n$ on $(\g,[\cdot,\cdot])$ given by
		$$n(X_1)=X_1,\quad n(X_2)=X_1+X_2,\quad n(X_3)=X_1+X_3,\quad n(X_4)=X_2-X_3+X_4\,,$$
		which defines the deformed Lie bracket
		\[
		[X_1,X_4]_n=X_1,\quad [X_2,X_4]_n=X_1,\quad [X_3,X_4]_n=X_2\,.
		\]
	Note that $n$ is not $\ad$-equivariant as
$n[X_3,X_4]\neq[X_3,nX_4] \,.$
		\noindent There is an $r$-matrix $r\in\wedge ^2\g$ on $\g$ 
		$$r=X_2 \wedge X_3-X_1\wedge X_4\,.$$
		\noindent 
		Thus $(\mathfrak g,[\cdot,\cdot],\delta_r)$ defined by $r$ is a Lie bialgebra, and the non-zero brackets are
		\[
		[X^1,X^2]_r=-X^3,\quad [X^1,X^4]_r=X^4\,.
		\]
        	The dual map $^tn$ is defined by
		\[
		{}^tn(X^1)=X^1+X^2+X^3,\quad {}^tn(X^2)=X^2+X^4,\quad {}^tn(X^3)=X^3-X^4,\quad {}^tn(X^4)=X^4\,.
		\]
		One can see that $n\circ \rs=\rs\circ {}^tn\,$, and $[\cdot,\cdot]_{r}^{{}^tn}=[\cdot,\cdot]_{nr}$ and so $C(r,n)\equiv 0$. Thus, $n\rs$ is also an $r$-matrix, reads as
	$n\rs=-X_1 \wedge X_2+X_1\wedge X_3-X_1\wedge X_4+X_2\wedge X_3\,.$ 
   It defines the Lie bracket
 \[
		[X^1,X^2]_{nr}=-X^3,\quad [X^1,X^3]_{nr}=-X^4,\quad [X^1,X^4]_{nr}=X^4\,.
		\]
\noindent The dual map ${{}^tn}$ is a Nijenhuis structure on the dual Lie algebra $(\g^*,[\cdot,\cdot]_r)$, and according  to the Proposition \ref{cobocom}, $(\mathfrak g,[\cdot,\cdot],\delta_r,n)$ is a coboundary NL bialgebra and it generates a hierarchy of Lie bialgebras. Note that the dual map ${}^tn$ is not $\ad^*$-equivariant as
		${}^tn[X^1,X^3]_r\neq [X^1,{}^tnX^3]_r\,.$
	\end{example}
	\subsection{Weak NL bialgebra on Euler-top system}\label{sec6}
	In the case of weak NL bialgebras, the existence of a hierarchy of iterated Lie brackets, 
$1$-cocycles, and subsequently, Lie bialgebras, is not guaranteed. In the following, we illustrate this situation using a weak NL bialgebra that serves as the underlying algebraic structure of a well-known dynamical system.

	\noindent 	We consider a three dimensional Lie bialgebra, with the commutators
	\[
	[X_1,X_2]= -X_2,\quad [X_1,X_3]=- X_3, \quad
	[X_2,X_3]= 0,\footnote{$(\mathfrak g,[\cdot,\cdot])$ is the book Lie algebra.}
	\]
	\[
	[X^1,X^2]_{g^*}=-X^3,\quad [X^1,X^3]_{g^*}= X^2,\quad [X^2,X^3]_{g^*}=-X^1,\footnote{$(\mathfrak g^*,[\cdot,\cdot]_{\g^*})$ is the Lie algebra $\so(3)$.}
	\]	
	the corresponding $1$-cocycle is $$\delta(X_1) =-X_2\wedge X_3,\quad  \delta(X_2)=  X_1\wedge X_3, \quad
	\delta(X_3) = -X_1\wedge X_2.$$
	\noindent The non-zeros coadjoint representation on the dual Lie algebra $\g^*$ of the Lie algebra $(\g,[\cdot,\cdot])$ are
	\[
	\begin{array}{rclrclrclrcl}
		&&	{\ad_{X_1}^*}X^2=X^2,\quad && 	{\ad_{X_1}^*}X^3=X^3,\quad && 	{\ad_{X_2}^*}X^2=-X^1,\quad 	&&{\ad_{X_3}^*}X^3=-X^1\,.\\[4pt]
	\end{array}		
	\]
	\noindent 	 There is a Nijenhuis operator $n:\g\to \g$ characterized by
	\[
	n(X_1)=X_3,\quad n(X_2)=X_2,\quad n(X_3)=-X_1-X_2+X_3.
	\]
	which defines a deformed Lie bracket
	\[
	[X_1,X_3]_n=-X_1,\quad [X_2,X_3]_n=-X_2.
	\]
	One can see that the $1$-cocycle $\delta$ is also a $1$-cocycle (denoted by $\delta^n$) in the deformed cohomology induced by $(\ad^n)^{(2)}$, on the Lie algebra $(\g,[\cdot,\cdot]_n)$. It defines a new Lie bialgebra with pair $((\g,[\cdot,\cdot]_n),(\g^*,[\cdot,\cdot]_{\g^*}))$. The non-zero elements of the coadjoint representation on the dual Lie algebra $\g^*$ of the Lie algebra $(\g,[\cdot,\cdot]_n)$ are as follows:
	\[
	\begin{array}{rclrclrcl}
		&&	(\ad^n_{X_1})^*X^1=X^3,\quad && 	(\ad^n_{X_2})^*X^2=X^3,\quad 	&&(\ad^n_{X_3})^*X^1=-X^1,\quad (\ad^n_{X_3})^*X^2=-X^2\,.\\[4pt]
	\end{array}		
	\]
	\noindent	The transpose map $^tn:\g\to\g$ where
	$
	^tn(X^1)=-X^3$, $^tn(X^2)=X^2-X^3$,	$^tn(X^3)=X^1+X^3$,
	defines a Lie bracket on $\g^*$ by
	\[
	[X^1,X^2]^{^tn}=-X^2,\quad [X^1,X^3]^{^tn}= X^3,\quad [X^2,X^3]^{^tn}=-2X^1\,.
	\footnote{$(\mathfrak g^*,[\cdot,\cdot]^{^tn})$ is the Lie algebra $\sl(2, \mathbb{R})$.}
	\]
	\noindent   Therefore, by definition, $(\g,[\cdot,\cdot],\delta,n)$, to this point, is an almost NL bialgebra. This structure serves as the underlying algebraic framework for a well-known dynamical system, particular case of the ${\mathfrak so}(3)$ Euler-top \cite{BaMaRa, We},
	\begin{equation}\label{system 3}
		\begin{array}{rcl}
			\dot{x_1}&=&x_2^2-x_3^2,\\
			\dot{x_2}&=& x_1(2x_3-x_2)\\
			\dot{x_3}&=&x_1(x_3-2x_2)\,.
		\end{array}
	\end{equation}
	\noindent This system is bi-Hamiltonian\footnote{A dynamical system is a Poisson bi-Hamiltonian if there exist two compatible Poisson structures whose Hamiltonian vector fields coincide.} with respect to linear Poisson structures
	\[
	\begin{array}{rclrclrcl}
		&&	\{x_1,x_2\}=-x_3,\quad  &&\{x_1,x_3\} =   x_2, \quad  && \{x_2,x_3\}=-x_1,\\[3pt]
		&&	\{x_1,x_2\}_{^tn}=-x_2, \quad  &&\{x_1,x_3\}_{^tn} = x_3, \quad && \{x_2,x_3\}_{^tn} =-2x_1.
	\end{array}
	\]
	These Linear Poisson structures $\{\cdot,\cdot\}$ and  $\{\cdot,\cdot\}_{^tn}$ correspond to the Lie algebras $[\cdot,\cdot]_{\g^*}$ and  $[\cdot,\cdot]^{^tn}$.
	
	\noindent By Proposition \ref{prop-equiv}, the deformation of $1$-cocycle $\delta$ by means of $^tn$ given by
	\[
	\delta_{^tn}(X_1)=-2X_2\wedge X_3,\quad \delta_{^tn}(X_2)=-X_1\wedge X_2,\quad \delta_{^tn}(X_3)=X_1\wedge X_3\,,\quad 
	\]
 is also a $1$-cocycle in the cohomology defined by the initial Lie algebra. Here $^tn$ is not a Nijenhuis operator on $\g^*$ but the Nijenhuis torsion $\lcf ^tn,^tn\rcf$ is a $2$-cocycle in the cohomology induced by the adjoint representation  of $(\mathfrak g,[\cdot,\cdot]$). Theretofore, we have a new Lie bialgebra $((\g,[\cdot,\cdot]),(\g^*,[\cdot,\cdot]^{^tn}))$, where $[\cdot,\cdot]^{{}^tn}:=\delta_{{}^tn}
	(\cdot,\cdot)$. One can see that $n$ is $\ad^{(2)}$-equivariant, that is $C(\delta,n)\equiv 0$, where
	\[
	\begin{array}{rclrclrc}
		&&\ads^*_{\,X_1}X^1=X^3,\quad
		&&\ads^*_{\,X_2}X^1=0,\quad
		&&\ads^*_{\,X_3}X^1=-X^1,	\\[3pt]
		
		&&\ads^*_{\,X_1}X^2=0,\quad
		&&\quad \quad \quad \ads^*_{\,X_2}X^2=X^1+X^3,\quad
		&&\ads^*_{\,X_3}X^2=-X^1,		
		\\[3pt]
		&&\ads^*_{\,X_1}X^3=X^1,\quad 
		&&\ads^*_{\,X_2}X^3=0, \quad &&\quad \ads^*_{\,X_3}X^3=X^1+X^3\,.\\[3pt]
	\end{array}
	\] 
 Or equivalently, according to the Theorem \ref{nnt}, one can check $\delta_{^tn}$ is a  $1$-cocycle in the deformed cohomology induced by $(\ad^n)^{(2)}:\g\times \wedge ^2\g\to  \wedge ^2\g $ on the Lie algebra $(\g,[\cdot,\cdot]_n)$. This establishes $(\mathfrak g,[\cdot,\cdot],\delta,n)$ as a weak NL bialgebra. Therefore, there is no guarantee for having a hierarchy of deformed Lie bialgebras by means of $(n_i,(^tn)^j)$. However, since $C(\delta,n)\equiv 0$, it implies that the double deformation $\delta_{^tn}^n$ is also $1$-cocycle. This allows us to construct a new Lie bialgebra  $((\mathfrak g,[\cdot,\cdot]_n),(\mathfrak g^*,[\cdot,\cdot]^{^tn}))$. Furthermore, according to Corollary \ref{doubledef}, $\delta_{{}^tn^2}$ and $\delta^{n^2}$ are also $1$-cocycles.
	
	\noindent Note that the recursive relations (\ref{nk}) and (\ref{Dnii}) do not hold for the deformed Lie brackets; however, the deformed bracket  
	\[	[X^1,X^2]^{{}^tn^2}=-2X^2,\quad [X^1,X^3]^{{}^tn^2}=X^1-2X^2+2X^3,\quad [X^2,X^3]^{{}^tn^2}=-2X^1-X^2\,,
	\]
    defined by 
    ${}^tn^2(X^1)=-X^1-X^3,\quad {}^tn^2(X^2)=-X^1+X^2-2X^3,\quad {}^tn^2(X^3)=X^1\,,	$	 
	is a Lie bracket on $\g^*$. The deformed $1$-cochain
	\[
	\begin{array}{rcl}			\delta_{{}^tn^2}(X_1)&=& X_1\wedge X_3-2X_2\wedge X_3,\\		\delta_{{}^tn^2}(X_2)&=& -2X_1\wedge X_2-2X_1\wedge X_3-X_2\wedge X_3,\\		\delta_{{}^tn^2}(X_3)&=& 2X_1\wedge X_3\,,		\end{array}	\]
	is a $1$-cocycle with respect to the adjoint representation of the Lie bracket $[\cdot,\cdot]$. Then the pairs $((\g,[\cdot,\cdot]_{n^2}),(\g^*,[\cdot,\cdot]))$ and $((\g,[\cdot,\cdot]),(\g^*,[\cdot,\cdot]^{{}^tn^2}))$ are also Lie bialgebras.

	\section{Concluding Remarks}
	
		In this paper we applied the theory of Poisson-Nijenhuis structures to Lie bialgebras. By examining the roles of Nijenhuis operators and the classical Yang-Baxter equation, we uncovered new insights into the hierarchical structures of Lie bialgebras. A Lie bialgebra $(\g,\g^*)$ equipped with a Nijenhuis structure $n$ on $\g$ that satisfies specific compatibility conditions is termed a (weak) NL bialgebra. We have meticulously investigated the necessary conditions for establishing a compatible hierarchy of Lie bialgebras. NL bialgebras provide profound insights into the underlying algebraic frameworks of certain dynamical systems. 
We illustrated this concept using the Euler-top system, demonstrating how its underlying algebraic structure constitutes a weak NL bialgebra. 

\noindent  The significance and utility of (weak) NL bialgebras can be observed in the following contexts:
\subsection{NL bialgebras and Poisson Lie groups}	There is a one-to-one correspondence between Lie bialgebras on a Lie algebra $\g$ and Poisson-Lie groups on the corresponding connected simply-connected Lie group $G$. Recall that Poisson-Lie groups are Lie groups equipped with a multiplicative Poisson structure i.e. the multiplication is a Poisson epimorphism. If $(\g, \g^*)$ is a Lie bialgebra and $G$ is the corresponding  connected simply-connected Lie group with Lie algebra $\mathfrak g$, then $G$ admits a multiplicative Poisson structure $\Pi$  such that  $[\cdot,\cdot]^*=(\d_e\Pi)^t,$ where $\d_e\Pi:\mathfrak g\to \wedge^2\mathfrak g$ is the linear map defined by $d_e\Pi(\xi)=({\mathcal L}_{\bar{\xi}}\Pi)(e)$, $\bar{\xi}$ being a vector field on $G$ such that $\bar{\xi}(e)=\xi$ (see \cite{Va}, \cite{VG}). Conversely, an adjoint $1$-cocycle $\delta:\g\to\wedge^2\g$ whose dual map
	satisfies the Jacobi identity induces a unique multiplicative Poisson structure on a connected simply-connected Lie group $G$ integrated the Lie algebra $\g$. 
	
	\noindent Building on the previous discussion, investigating the global object whose infinitesimal counterpart is a (weak) NL bialgebra presents an intriguing problem. Ongoing research in this area aims to address these questions.

	\subsection{NL bialgebras and quantum groups} The quantization of Lie bialgebras into Hopf algebras is a key process in constructing quantum groups. Reshetikhin  established that every finite-dimensional Lie bialgebra $(\g,\g^*)$ over a field $K$ of characteristic zero admits a quantization \cite{Resh}, (see also \cite{ChPr, VG1}). This means there exists a corresponding $\hbar$-deformation $U_{\hbar}(\g)$ of its universal enveloping algebra $U(\g)$. The quantization process can be described using the Baker-Campbell-Hausdorff series $H(\xi_1,\xi_2)$ for the Lie algebra $\g$,
\[
H(\xi_1,\xi_2)=\xi_1+\xi_2+\frac{1}{2}[\xi_1,\xi_2]+\frac{1}{12}[\xi_1,[\xi_1,\xi_2]]-\frac{1}{12}[\xi_2,[\xi_1,\xi_2]]+\cdots
\]  
\noindent   Nijenhuis structures, particularly (weak) NL bialgebras, can provide valuable insights into the algebraic properties of quantum groups. The deformed Lie bialgebras that emerge in the hierarchy serve as infinitesimal structures that describe quantum groups. For instance, when $n$ is $\ad$-equivariant, the Baker-Campbell-Hausdorff series $H_n$ for the Lie algebra $(\g,[\cdot,\cdot]_n)$ simplifies to $H_n(\xi_1,\xi_2)=H(n\xi_1,\xi_2)$. This observation can be used to describe the $\hbar$-deformation of an N-deformation on $\g$, and its relationship with the $\hbar$-deformation of the Lie bialgebra itself.

	\medskip
	\noindent {\bf Acknowledgments.} 
	The author acknowledges the ``Cercetare  postdoctoral\u a avansat\u a" funded by the West University of Timi\c soara, Romania, the financial support from the Spanish Ministry of Science and Innovation under grants PID2022-137909NB-C22, and the Max-Plank institute for Mathematics in Bonn, MPIM-Bonn-2023, where a part of this work has been done. The author is grateful to Juan Carlos Marrero and Edith Padr\'on for their valuable discussions. In addition, the author wishes to express sincere gratitude to the referee for constructive feedback and recommendations, which greatly contributed to the improvement of this paper.
	
\end{document}